\documentclass{amsart}

\pdfoutput=1

\usepackage{etex,amsfonts, amsmath, amsthm, amssymb, stmaryrd, epsfig, graphics,
  psfrag, latexsym, mathtools, mathrsfs,enumitem, subfig,
  longtable, booktabs, yfonts, centernot,mathscinet} 
\usepackage{xcolor}
\usepackage[all,cmtip]{xy}
\usepackage[percent]{overpic}
\definecolor{darkblue}{rgb}{0,0,0.4} 
 \usepackage[colorlinks=true, citecolor=darkblue, filecolor=darkblue, linkcolor=darkblue,urlcolor=darkblue]{hyperref} 
\usepackage[all]{hypcap}
\usepackage{xr-hyper}

\usepackage[color=blue!20!white]{todonotes}
\makeatletter
\providecommand\@dotsep{5}
\def\listtodoname{List of Todos}
\def\listoftodos{\@starttoc{tdo}\listtodoname}
\makeatother

\graphicspath{{draws/}{}}

\numberwithin{equation}{section}

\newtheorem{thm}{Theorem}
\newtheorem{theorem}[thm]{Theorem}

\newtheorem{lem}{Lemma}[section]

\newtheorem{proposition}[lem]{Proposition}
\theoremstyle{definition}     

\newtheorem{question}[lem]{Question}
 
\newtheorem{definition}[lem]{Definition}

\theoremstyle{remark}     

\newtheorem{remark}[lem]{Remark}

\newtheorem{example}[lem]{Example}

\numberwithin{figure}{section}
\numberwithin{table}{section}

\newcommand{\Section}[1]{Section~\ref{sec:#1}}

\newcommand{\Theorem}[1]{Theorem~\ref{thm:#1}}
\newcommand{\Definition}[1]{Definition~\ref{def:#1}}
\newcommand{\Remark}[1]{Remark~\ref{rem:#1}}

\newcommand{\Proposition}[1]{Proposition~\ref{prop:#1}}

\newcommand{\Example}[1]{Example~\ref{exam:#1}}

\newcommand{\Diagram}[2][{}]{Diagram#1~(\ref{eq:#2})}

\newcommand{\Z}{\mathbb{Z}}
\newcommand{\F}{\mathbb{F}}

\newcommand{\mc}{\mathcal}
\newcommand{\mb}{\mathbb}

\newcommand{\mf}{\mathfrak}

\newcommand{\wt}{\widetilde}

\newcommand{\del}{\partial}

\renewcommand{\emptyset}{\varnothing}

\newcommand{\card}[1]{\left\vert{#1}\right\vert}
\newcommand{\Set}[2]{\{#1\mid #2\}}

\newcommand{\si}{\sigma}

\newcommand{\from}{\colon}
\newcommand{\into}{\hookrightarrow}

\newcommand{\isomorphicto}{\stackrel{\cong}{\longrightarrow}}

\renewcommand{\th}{^{\text{th}}}

\DeclareMathOperator{\Ob}{Ob}
\DeclareMathOperator{\Id}{Id}
\DeclareMathOperator{\Sq}{Sq}

\DeclareMathOperator{\Hom}{Hom}

\newcommand{\Kh}{\mathit{Kh}}
\newcommand{\KA}[1]{\mathbb{H}^1}
\newcommand{\rKh}{\widetilde{\Kh}}
\newcommand{\KhCx}{\mc{C}_{\mathit{Kh}}}
\newcommand{\rKhCx}{\widetilde{\mathcal{C}}_{\mathit{Kh}}}

\newcommand{\Realize}[2][{}]{|#2|_{#1}}
\newcommand{\CRealize}[2][{}]{\|#2\|_{#1}}

\newcommand{\op}{\mathrm{op}}

\newcommand{\gr}{\mathrm{gr}}

\newcommand{\KhSpace}{\mathcal{X}_\mathit{Kh}}
\newcommand{\rKhSpace}{\widetilde{\mathcal{X}}_\mathit{Kh}}

\newcommand{\CubeCat}[1]{\underline{2}^{#1}}

\newcommand{\Permu}[1]{{\Pi}_{#1}}
\newcommand{\permu}[1]{{\mathcal{S}}_{#1}}

\newcommand{\co}{\colon}
\newcommand{\bdy}{\partial}

\newcommand{\ZZ}{\mathbb{Z}}

\DeclareMathOperator{\Span}{span}

\newcommand{\Dual}[1]{(#1)^*}

\newcommand{\cellC}[1][\bullet]{\wt{C}^{\mathit{cell}}_{#1}}

\newcommand{\BurnsideCat}{\mathscr{B}}
\newcommand{\CCat}[1]{\CubeCat{#1}}
\newcommand{\cmorph}[2]{\varphi_{#1,#2}}
\newcommand{\AbelianGroups}{\mathsf{Ab}}
\newcommand{\AbFunc}{\mathcal{A}}
\newcommand{\thic}[1]{\widehat{#1}}
\newcommand{\thicf}[2][{}]{\widehat{#2}_{#1}}

\newcommand{\Spectra}{\mathscr{S}}
\newcommand{\Complexes}{\mathsf{Kom}}
\newcommand{\Total}{\mathsf{Tot}}
\newcommand{\CW}{\mathsf{CW}}
\newcommand{\CWtoSpec}[1]{(\Sigma^{\infty}{#1})}

\newcommand{\Sets}{\mathsf{Sets}}

\newcommand{\KhFunc}{F_{\Kh}}

\newcommand{\HKKfun}{F_{\mathit{HKK}}}

\DeclareMathOperator{\hocolim}{hocolim}

\newcommand{\Ab}{\mathit{Ab}}
\newcommand{\KhAbFunc}{F_{\Kh,\Ab}}

\begin{document}

 
\title{The cube and the Burnside category}

\author{Tyler Lawson}
\thanks{TL was supported by NSF Grant DMS-1206008.}
\email{\href{mailto:tlawson@math.umn.edu}{tlawson@math.umn.edu}}
\address{Department of Mathematics, University of Minnesota, Minneapolis, MN 55455}

\author{Robert Lipshitz}
\thanks{RL was supported by NSF Grant DMS-1149800.}
\email{\href{mailto:lipshitz@uoregon.edu}{lipshitz@uoregon.edu}}
\address{Department of Mathematics, University of Oregon, Eugene, OR 97403}

\author{Sucharit Sarkar}
\thanks{SS was supported by NSF Grant DMS-1350037.}
\email{\href{mailto:sucharit@math.princeton.edu}{sucharit@math.princeton.edu}}
\address{Department of Mathematics, Princeton University, Princeton, NJ 08544}



\date{May 1, 2015}

\begin{abstract}
  In this note we present a combinatorial link invariant that
  underlies some recent stable homotopy refinements of Khovanov
  homology of links. The invariant takes the form of a functor between
  two combinatorial 2-categories, modulo a notion of stable
  equivalence. We also develop some general properties of such
  functors.
\end{abstract}

\maketitle


\section{Introduction}\label{sec:intro}

In the last five years, two stable homotopy refinements of the
Khovanov homology of links were
introduced~\cite{RS-khovanov,HKK-Kh-htpy}. In a recent paper, we
showed that these two refinements agree, and in the process gave a
simplified construction of these
invariants~\cite{LLS-khovanov-product}. That paper still involved a
certain amount of topology. In this note, we present a
combinatorial link invariant from which one can extract the Khovanov
homotopy types, with the goal of making aspects of our earlier work
more broadly accessible.

We will study functors from the cube category $\CCat{n}$, the small
category associated to the partially ordered set $\{0,1\}^n$
(\Section{cube}), to the Burnside category $\BurnsideCat$, which is
the 2-category whose objects are finite sets and morphisms are finite
correspondences (\Section{burnside}). The original construction of
Khovanov homology~\cite{Kho-kh-categorification} using Kauffman's
$n$-dimensional cube of resolutions~\cite{Kau-knot-resolutions} of
an $n$-crossing link diagram $K$ can be generalized to construct a
2-functor
\[
F_{\Kh}(K)\from\CCat{n}\to\BurnsideCat.
\]
See \Section{khovanov-functor} for the exact definition, and for some
additional grading shifts that are involved. (Such functors also arise
in other contexts, such as from a simplicial complex with $n$
vertices; see \Example{delta-cx-functor}.)

To any such functor $F\from \CCat{n}\to\BurnsideCat$, one can
associate a chain complex $\Total(F)\in\Complexes$, the
\emph{totalization} of $F$. Indeed, this
construction can be thought of a functor
\[
\Total\from\BurnsideCat^{\CCat{n}}\to\Complexes.
\]
The totalization of the functor $F_{\Kh}$ recovers the dual of the
Khovanov complex of $K$:
\[
\Dual{\KhCx(K)}=\Total(F_{\Kh}).
\]
(The dual only appears to maintain consistency with earlier
conventions~\cite{RS-khovanov}.)

Given a functor $F\from\CCat{n}\to\BurnsideCat$ and a sufficiently
large $\ell\in\Z$, one can construct a based CW complex
$\CRealize[\ell]{F}\in\CW$. The dependence on $\ell$ is simple:
$\CRealize[{\ell+1}]{F}$ is merely the reduced suspension
$\Sigma\CRealize[\ell]{F}$. The reduced cellular chain complex of
$\CRealize[\ell]{F}$ is the earlier chain complex up to a
shift:
\[
\cellC(\CRealize[\ell]{F})=\Sigma^\ell\Total(F).
\]
Similarly, to any natural transformation $\eta\from F\to F'$ between
two such 2-functors $F,F'\from\CCat{n}\to\BurnsideCat$, 
one can associate a pointed cellular map
$\CRealize[\ell]{F}\to\CRealize[\ell]{F'}$ that induces the map
$\Sigma^\ell\Total(\eta)$ on the reduced cellular chain complexes.
Unfortunately, the definitions of these spaces and these maps depend
on certain choices, and therefore the construction is not strictly
functorial.  We may eliminate these choices if we are willing to work
with spectra. Indeed, if $\Spectra$ is any reasonable category of
spectra,
then there is a canonical functor
\[
\Realize{\cdot}\from\BurnsideCat^{\CCat{n}}\to\Spectra
\]
so that for any $F\from\CCat{n}\to\BurnsideCat$, $\Realize{F}$ is
isomorphic to $\Sigma^{-\ell}\CWtoSpec{\CRealize[\ell]{F}}$, the
$\ell\th$ desuspension of the suspension spectrum of
$\CRealize[\ell]{F}$. Up to a grading shift, the Khovanov homotopy
type associated to the $n$-crossing link diagram $K$ is
\[
\KhSpace(K)=\Realize{F_{\Kh}},
\]
and therefore, its cohomology ${H}^\bullet(\KhSpace(K))$ is
isomorphic to the Khovanov homology $\Kh(K)=H_\bullet(\KhCx(K))$.  

In this note, we only sketch the construction of $\CRealize[\ell]{F}$ and
only hint at the construction of $\Realize{F}$ (\Section{spaces}).  We
will focus on functors from cubes to the Burnside category and
define an equivalence relation on such functors, generated by
the following two relations:
\begin{enumerate}[leftmargin=*]
\item If $\eta\from F\to F'$ is a natural transformation between two
  functors $F,F'\from\CCat{n}\to\BurnsideCat$ and the induced map
  $\Total(\eta)\from\Total(F)\to\Total(F')$ is a chain homotopy
  equivalence, then $F$ is stably equivalent to $F'$.
\item If $\iota\from\CCat{n}\into\CCat{N}$ is a face inclusion, and
  $F\from\CCat{n}\to\BurnsideCat$, then there is an induced functor
  $F_{\iota}\from\CCat{N}\to\BurnsideCat$, induced by
  $F_{\iota}\circ\iota=F$ and for any
  $v\in\CCat{N}\setminus\iota(\CCat{n})$,
  $F_{\iota}(v)=\emptyset$. The functor $F_{\iota}$ is stably
  equivalent to $F$ (up to a grading shift).
\end{enumerate}
See \Definition{equiv-functors} for the precise version. This notion
of \emph{stable equivalence} ensures that if $F,F'$ are stably equivalent
functors, then $\Realize{F}$ and $\Realize{F'}$ are homotopy
equivalent spectra; or equivalently, $\CRealize[\ell]{F}$ and
$\CRealize[\ell]{F'}$ are stably homotopy equivalent CW complexes.

The main result of this paper is that the Khovanov functor $\KhFunc$
is a link invariant. Namely, if $K$ and $K'$ are isotopic link
diagrams, then the Khovanov functors $F_{\Kh}(K)$ and $F_{\Kh}(K')$
are stably equivalent (\Theorem{main}, \Section{khovanov-functor}).

\thinspace
\noindent\emph{Acknowledgments}. We thank the referee for comments on
a draft of this paper. The second author also thanks the organizers of
the conference ``Categorification in Algebra, Geometry and Physics''
for creating a stimulating environment in which to discuss these
results.

\section{The cube}\label{sec:cube}

The one-dimensional cube is $\CCat{}=\{0,1\}$. It can be viewed as a
partially ordered set by declaring $1>0$. It can also be viewed as a
category with a single non-identity morphism from $1$ to
$0$. There is a notion of grading, where we declare the grading of
$v\in\CCat{}$ to be $\card{v}=v$.

The $n$-dimensional cube is the Cartesian product
$\CCat{n}=\{0,1\}^n$. It has an induced partial order, where
\[
(u_1,\dots,u_n)\geq (v_1,\dots,v_n) \text{ if and only if }\forall
i(u_i\geq v_i),
\]
and an induced categorical structure: the morphism set
$\Hom_{\CCat{n}}(u,v)$ has a single element if $u\geq v$, and is empty
otherwise. 
For $u\geq v$, we will write $\cmorph{u}{v}$
to denote the unique morphism in $\Hom_{\CCat{n}}(u,v)$. Finally, there is an
induced grading, which is simply the $L^1$-norm:
\[
\card{v}=\sum_i v_i.
\]
For convenience, we will write $u\geq_k v$ if $u\geq v$ and
$\card{u}-\card{v}=k$; and we will sometimes write
$\xymatrix@C=1.5ex{u\ar@{-{*}}[r]&v}$ if $u\geq_1 v$.

We will need the following \emph{sign assignment} function.
\begin{definition}\label{def:sign-assign}
  Given $u=(u_1,\dots,u_n)\geq_1 v=(v_1,\dots, v_n)$, let $k$ be
  the unique element in $\{1,\dots,n\}$ satisfying $u_k>v_k$, and define
\[
s_{u,v}=\sum_{i=1}^{k-1}u_i \pmod 2.
\]
\end{definition}

\section{The Burnside category}\label{sec:burnside}

For us, the Burnside category $\BurnsideCat$ is the 2-category of
finite sets, finite correspondences, and bijections of
correspondences.  The objects $\Ob(\BurnsideCat)$ are finite
sets; for any two objects $A,B$, the morphisms
$\Hom_{\BurnsideCat}(A,B)$ are the finite correspondences (or spans)
from $A$ to $B$, that is, triples $(X,s,t)$ where $X$ is a finite set,
and $s\from X\to A$ and $t\from X\to B$ are set maps, called the 
\emph{source map} and the \emph{target map}, respectively.
We usually denote such correspondences by
diagrams \( \vcenter{\xymatrix@R=1ex{&X\ar[dl]_-s\ar[dr]^-t\\A&&B}}
\); and we often drop $s$ and $t$ from the notation if they are
irrelevant to the discussion. The identity morphism
$\Id_A\in\Hom_{\BurnsideCat}(A,A)$ is the correspondence \(
\vcenter{\xymatrix@R=1ex{&A\ar[dl]_-{\Id}\ar[dr]^-{\Id}\\A&&A}}
\). Composition is given by fiber product.  That is, for $X$ in
$\Hom_{\BurnsideCat}(A,B)$ and $Y$ in $\Hom_{\BurnsideCat}(B,C)$, the
composition $Y\circ X$ is defined to be the fiber product $Y\times_B
X=\Set{(y,x)\in Y\times X}{s(y)=t(x)}$ in $\Hom_{\BurnsideCat}(A,C)$:
\[
\xymatrix@R=1ex{&&Y\times_B X\ar[dl]\ar[dr]&\\&X\ar[dl]\ar[dr]&&Y\ar[dl]\ar[dr]\\A&&B&&C.}
\]
For morphisms $X,Y$ in $\Hom_{\BurnsideCat}(A,B)$, define the
$2$-morphisms from $X$ to $Y$ to be the bijections $X\isomorphicto Y$
so that the following diagram commutes:
\[
\xymatrix@R=1ex{&X\ar[dddl]\ar[dddr]\ar[dd]\\ \\&Y\ar[dl]\ar[dr]\\A&&B.}
\]

The Burnside category $\BurnsideCat$ is self-dual, in the sense of
being isomorphic to its own opposite $\BurnsideCat^\op$. The
isomorphism preserves objects, and sends
$(X,s,t)\in\Hom_{\BurnsideCat}(A,B)$ to
$(X,t,s)\in\Hom_{\BurnsideCat^{\op}}(B,A)$.  The category of finite
sets is essentially a subcategory of $\BurnsideCat$ since we can view
any set map $f\from A\to B$ as the correspondence \(
\vcenter{\xymatrix@R=1ex{&A\ar[dl]_-{\Id}\ar[dr]^-f\\A&&B}} \). (This
statement is not literally true because composition is only preserved
up to $2$-isomorphism.)

We will need the following functor from $\BurnsideCat$ to the category
of Abelian groups $\AbelianGroups$.
\begin{definition}\label{def:abfunc}
  Define the functor $\AbFunc\from\BurnsideCat\to\AbelianGroups$ as
  follows. For $A\in\Ob(\BurnsideCat)$, define $\AbFunc(A)=\Z\langle
  A\rangle$, the free Abelian group with basis $A$. For a
  correspondence $(X,s,t)\in\Hom_{\BurnsideCat}(A,B)$, define the map
  $\AbFunc(X)\from\AbFunc(A)\to\AbFunc(B)$ by defining it on the basis
  elements $a\in A$ as
  \[
  \AbFunc(X)(a)=\sum_{b\in B} \card{s^{-1}(a)\cap t^{-1}(b)}b.
  \]
\end{definition}
If, in $\BurnsideCat$, we replace correspondences with isomorphism
classes of correspondences, we obtain an ordinary category. The
functor $\AbFunc$ identifies this category with a subcategory of
$\AbelianGroups$, whose objects are the finitely generated free
Abelian groups with a chosen basis and whose morphisms are represented
by matrices of nonnegative integers.

\section{Functors from the cube to \texorpdfstring{$\BurnsideCat$}{B}}\label{sec:functor}

The central objects that we will study are pairs $(F,r)$, where $F$ is
a functor $\CCat{n}\to\BurnsideCat$ and $r\in\ZZ$ is an integer; we
will use the notation $\Sigma^r F$ to denote such a pair, and call it
a \emph{stable functor}. When it is unlikely to cause confusion we will
suppress $\Sigma^r$ from the notation, and refer to a stable functor $F$.

The cube category $\CCat{n}$ is a strict category, while the Burnside
category $\BurnsideCat$ is a weak $2$-category, so we owe some
explanation regarding what we mean by a functor
$F\from\CCat{n}\to\BurnsideCat$. We will treat $\CCat{n}$ as a
$2$-category with no non-identity $2$-morphisms, and the functor $F$
will be a strictly unitary lax $2$-functor. (For general definitions,
see \cite[Definition~4.1 and Remark~4.2]{Benabou-other-bicategories}
which calls such functors \emph{strictly unitary homomorphisms}.)

\begin{definition}\label{def:cube2burnside-functor}
  A \emph{strictly unitary lax $2$-functor} $F$ from the cube category
  $\CCat{n}$ to the Burnside category $\BurnsideCat$ consists of the
  following data:
  \begin{enumerate}[leftmargin=*]
  \item a finite set $F(v)\in\Ob(\BurnsideCat)$ for every
    $v\in\{0,1\}^n$,
  \item a finite correspondence
    $F(\cmorph{u}{v})\in\Hom_{\BurnsideCat}(F(u),F(v))$ for every
    $u>v$, and
  \item a $2$-isomorphism $F_{u,v,w}\from
    F(\cmorph{v}{w})\times_{F(v)}F(\cmorph{u}{v})\to F(\cmorph{u}{w})$
    for every $u>v>w$,
  \end{enumerate}
  so that for all $u>v>w>z$, the following diagram commutes:
  \begin{equation}\label{eq:lax-functor-pentagon}
    \vcenter{\xymatrix@C=12ex{
        F(\cmorph{w}{z})\times_{F(w)}F(\cmorph{v}{w})\times_{F(v)}
        F(\cmorph{u}{v})\ar[r]^-{\Id\times F_{u,v,w}}\ar[d]_-{F_{v,w,z}\times\Id}&F(\cmorph{w}{z})\times_{F(w)}F(\cmorph{u}{w})\ar[d]^-{F_{u,w,z}}\\
        F(\cmorph{v}{z})\times_{F(v)}F(\cmorph{u}{v})\ar[r]_-{F_{u,v,z}}&F(\cmorph{u}{z}).
      }}
  \end{equation}
  (Here, $F(\cmorph{w}{z})\times_{F(w)}F(\cmorph{v}{w})\times_{F(v)}
  F(\cmorph{u}{v})$ denotes either
  $\big(F(\cmorph{w}{z})\times_{F(w)}F(\cmorph{v}{w})\big)\times_{F(v)}
  F(\cmorph{u}{v})$ or
  $F(\cmorph{w}{z})\times_{F(w)}\big(F(\cmorph{v}{w})\times_{F(v)}
  F(\cmorph{u}{v})\big)$, which are not the same but are canonically
  identified.)
\end{definition}

We will typically refer to strictly unitary lax $2$-functors simply as
\emph{$2$-functors} or even just \emph{functors}.

\begin{example}\label{exam:delta-cx-functor}
  Let $X$ be a finite $\Delta$-complex (as in
  \cite[Section~2.1]{Hatcher-top-book}) such that every $k$-simplex
  contains $(k+1)$ distinct vertices. (For example, $X$ could be a
  finite simplicial complex.) For each $k$, let $X(k+1)$ denote the
  set of $k$-simplices of $X$, and let $X(0)=\emptyset$.  Let
  $\{p_1,\dots,p_n\}=X(1)$ be the vertices of $X$.  To this we
  associate the stable functor $\Sigma^{-1}F^X$ where \(
  F^X\from\CCat{n}\to\BurnsideCat \) is given by
  \begin{align*}
    F^X(v)&=\Set{\Delta\in X(\card{v})}{\forall
      i\big((v_i=1)\iff(p_i\in\Delta)\big)} &\text{for every
      $v\in\{0,1\}^n$,}\\
    F^X(\cmorph{u}{v})&=\Set{(\Delta_v,\Delta_u)\in
      F^X(v)\times F^X(u)}{\Delta_v\subset\Delta_u}
    &\text{for every $u>v$ in $\{0,1\}^n$.}
  \end{align*}
  The source and target maps for the correspondences are the two
  projection maps. For every $u>v>w$, the $2$-isomorphisms $F^X_{u,v,w}$
  are uniquely determined, and the uniqueness forces
  \Diagram{lax-functor-pentagon} to commute.
\end{example}

Like the name, the data for a strictly unitary lax $2$-functor
$\CCat{n}\to\BurnsideCat$ might seem unwieldy, but fortunately there
is a smaller formulation. Consider the following three pieces of data:
\begin{enumerate}[label=(D-\arabic*),ref=(D-\arabic*)]
\item\label{data:set} for every vertex $v\in\{0,1\}^n$ of the cube,
  a finite set $F(v)\in\Ob(\BurnsideCat)$,
\item\label{data:correspondence} for every edge
  $\xymatrix@C=1.5ex{u\ar@{-{*}}[r]&v}$ of the cube, a finite
  correspondence $F(\cmorph{u}{v})\in\Hom_{\BurnsideCat}(F(u),F(v))$, and
\item\label{data:bijection} for every two-dimensional face
  $\vcenter{\xymatrix@R=0ex@C=1ex{&v\ar@{-{*}}[dr]\\u\ar@{-{*}}[ur]\ar@{-{*}}[dr]&&w\\&v'\ar@{-{*}}[ur]}}$
  of the cube, a $2$-morphism \[F_{u,v,v',w}\from
  F(\cmorph{v}{w})\times_{F(v)} F(\cmorph{u}{v})\to
  F(\cmorph{v'}{w})\times_{F(v')} F(\cmorph{u}{v'}),\]
\end{enumerate}
satisfying the following two conditions:
\begin{enumerate}[label=(C-\arabic*),ref=(C-\arabic*)]
\item\label{condition:matching} for every two-dimensional face
  $\vcenter{\xymatrix@R=0ex@C=1ex{&v\ar@{-{*}}[dr]\\u\ar@{-{*}}[ur]\ar@{-{*}}[dr]&&w\\&v'\ar@{-{*}}[ur]}}$,
  $F_{u,v',v,w}=F_{u,v,v',w}^{-1}$, and
\item\label{condition:hexagon} for every three-dimensional face
  $\vcenter{\xymatrix@R=0.8ex@C=1.5ex{&v\ar@{-{*}}[r]\ar@{-{*}}[dr]&w''\ar@{-{*}}[dr]\\u\ar@{-{*}}[ur]\ar@{-{*}}[r]\ar@{-{*}}[dr]&v'\ar@{-{*}}[ur]\ar@{-{*}}[dr]&w'\ar@{-{*}}[r]&z\\&v''\ar@{-{*}}[r]\ar@{-{*}}[ur]&w\ar@{-{*}}[ur]}}$,
  the following commutes:
  \[
  \xymatrix@C=10ex{
    F(\cmorph{w''}{z})\times_{F(w'')}F(\cmorph{v}{w''})\times_{F(v)}F(\cmorph{u}{v})\ar[r]^-{F_{v,w'',w',z}\times\Id}\ar[d]^-{\Id\times
      F_{u,v,v',w''}}&
    F(\cmorph{w'}{z})\times_{F(w')}F(\cmorph{v}{w'})\times_{F(v)}F(\cmorph{u}{v})\ar[d]_-{\Id\times
      F_{u,v,v'',w'}} \\
    F(\cmorph{w''}{z})\times_{F(w'')}F(\cmorph{v'}{w''})\times_{F(v')}F(\cmorph{u}{v'})\ar[d]^-{F_{v',w'',w,z}\times\Id}&F(\cmorph{w'}{z})\times_{F(w')}F(\cmorph{v''}{w'})\times_{F(v'')}F(\cmorph{u}{v''})\ar[d]_-{F_{v'',w',w,z}\times\Id}\\
    F(\cmorph{w}{z})\times_{F(w)}F(\cmorph{v'}{w})\times_{F(v')}F(\cmorph{u}{v'})\ar[r]^-{\Id\times
      F_{u,v',v'',w}}&F(\cmorph{w}{z})\times_{F(w)}F(\cmorph{v''}{w})\times_{F(v'')}F(\cmorph{u}{v''}).
  }
  \]
\end{enumerate}

A strictly unitary lax $2$-functor $F$ produces
data~\ref{data:set}--\ref{data:bijection} satisfying
conditions~\ref{condition:matching}--\ref{condition:hexagon}, by
simply declaring that $F_{u,v,v',w}=F_{u,v',w}^{-1} \circ F_{u,v,w}$.
Conversely:

\begin{proposition}\label{prop:redefine}
  Assume we are given data~\ref{data:set}--\ref{data:bijection}
  satisfying
  conditions~\ref{condition:matching}--\ref{condition:hexagon}. Then
  up to natural isomorphism, there is exactly one strictly unitary $2$-functor
  $F\from\CCat{n}\to\BurnsideCat$ that produces it.
\end{proposition}

\begin{proof}
  This is \cite[Lemma~2.12]{LLS-khovanov-product}, but for
  completeness, we give a proof. For both existence and uniqueness, we
  need the following facts about maximal chains on the cube
  $\CCat{n}$. Fix $u\geq_k v$, and consider maximal chains
  $\xymatrix@C=1.5ex{u=z_0\ar@{-{*}}[r]&\cdots\ar@{-{*}}[r]&z_i\ar@{-{*}}[r]&\cdots\ar@{-{*}}[r]&z_k=v}$. Then:
  \begin{enumerate}[leftmargin=*,label=(m-\arabic*),ref=(m-\arabic*)]
  \item\label{item:max-chain:2d} Any two such maximal chains
    $\mf{m}_1$ and $\mf{m}_2$ can be connected by a sequence of swaps
    across two-dimensional faces, that is, by a sequence of
    replacements of
    chains $\big(\xymatrix@C=1.5ex{\cdots\ar@{-{*}}[r]&z_{i-1}\ar@{-{*}}[r]&z_i\ar@{-{*}}[r]&z_{i+1}\ar@{-{*}}[r]&\cdots}\big)$
    by
    $\big(\xymatrix@C=1.5ex{\cdots\ar@{-{*}}[r]&z_{i-1}\ar@{-{*}}[r]&z'_i\ar@{-{*}}[r]&z_{i+1}\ar@{-{*}}[r]&\cdots}\big)$.
  \item\label{item:max-chain:3d} Any two such sequences $\mf{s}_1$ and
    $\mf{s}_2$ connecting any two such maximal chains $\mf{m}_1$ and
    $\mf{m}_2$ can be related by a sequence of moves of the following
    three types:
    \begin{enumerate}[leftmargin=0in,label=(\alph*),ref=(\alph*)]
    \item Replacing a sequence of the form
      $\{\ldots,\mf{m}_{\ell},\mf{m}'_{\ell},\mf{m}_{\ell},\ldots\}$
      with the sequence $\{\ldots,\mf{m}_{\ell},\ldots\}$.
    \item Exchanging a sequence of one of the following two forms for the other:
    \[
    \xymatrix@C=1.1ex@R=0ex{
      &&&&\vdots&&&&&&&&&&\vdots\\
      \cdots\ar@{-{*}}[r]&z_{i-1}\ar@{-{*}}[r]&z_i\ar@{-{*}}[r]&z_{i+1}\ar@{-{*}}[r]&\cdots\ar@{-{*}}[r]&z_{j-1}\ar@{-{*}}[r]&z_j\ar@{-{*}}[r]&z_{j+1}\ar@{-{*}}[r]&\cdot&&\cdot\ar@{-{*}}[r]&z_{i-1}\ar@{-{*}}[r]&z_i\ar@{-{*}}[r]&z_{i+1}\ar@{-{*}}[r]&\cdots\ar@{-{*}}[r]&z_{j-1}\ar@{-{*}}[r]&z_j\ar@{-{*}}[r]&z_{j+1}\ar@{-{*}}[r]&\cdots\\
      \cdots\ar@{-{*}}[r]&z_{i-1}\ar@{-{*}}[r]&z'_i\ar@{-{*}}[r]&z_{i+1}\ar@{-{*}}[r]&\cdots\ar@{-{*}}[r]&z_{j-1}\ar@{-{*}}[r]&z_j\ar@{-{*}}[r]&z_{j+1}\ar@{-{*}}[r]&\cdot\ar@{<->}[rr]&\txt{\phantom{xx}}&\cdot\ar@{-{*}}[r]&z_{i-1}\ar@{-{*}}[r]&z_i\ar@{-{*}}[r]&z_{i+1}\ar@{-{*}}[r]&\cdots\ar@{-{*}}[r]&z_{j-1}\ar@{-{*}}[r]&z'_j\ar@{-{*}}[r]&z_{j+1}\ar@{-{*}}[r]&\cdots\\
      \cdots\ar@{-{*}}[r]&z_{i-1}\ar@{-{*}}[r]&z'_i\ar@{-{*}}[r]&z_{i+1}\ar@{-{*}}[r]&\cdots\ar@{-{*}}[r]&z_{j-1}\ar@{-{*}}[r]&z'_j\ar@{-{*}}[r]&z_{j+1}\ar@{-{*}}[r]&\cdot&&\cdot\ar@{-{*}}[r]&z_{i-1}\ar@{-{*}}[r]&z'_i\ar@{-{*}}[r]&z_{i+1}\ar@{-{*}}[r]&\cdots\ar@{-{*}}[r]&z_{j-1}\ar@{-{*}}[r]&z'_j\ar@{-{*}}[r]&z_{j+1}\ar@{-{*}}[r]&\cdots\\
      &&&&\vdots&&&&&&&&&&\vdots&&&&&,
    }
    \]
    where $j-i\geq 2$.  (This corresponds to exchanging the order of
    swaps across two faces which share no edges.)
  \item Exchanging a sequence of one of the following two forms for the other:
    \[
    \xymatrix@C=1.5ex@R=0ex{
      &&&\vdots&&&&&&\vdots\\
      \cdots\ar@{-{*}}[r]&z_i\ar@{-{*}}[r]&z_{i+1}\ar@{-{*}}[r]&z''_{i+2}\ar@{-{*}}[r]&z_{i+3}\ar@{-{*}}[r]&\cdots&&\cdots\ar@{-{*}}[r]&z_i\ar@{-{*}}[r]&z_{i+1}\ar@{-{*}}[r]&z''_{i+2}\ar@{-{*}}[r]&z_{i+3}\ar@{-{*}}[r]&\cdots\\
      \cdots\ar@{-{*}}[r]&z_i\ar@{-{*}}[r]&z'_{i+1}\ar@{-{*}}[r]&z''_{i+2}\ar@{-{*}}[r]&z_{i+3}\ar@{-{*}}[r]&\cdots\ar@{<->}[rr]&\txt{\phantom{xx}}&\cdots\ar@{-{*}}[r]&z_i\ar@{-{*}}[r]&z_{i+1}\ar@{-{*}}[r]&z'_{i+2}\ar@{-{*}}[r]&z_{i+3}\ar@{-{*}}[r]&\cdots\\
      \cdots\ar@{-{*}}[r]&z_i\ar@{-{*}}[r]&z'_{i+1}\ar@{-{*}}[r]&z_{i+2}\ar@{-{*}}[r]&z_{i+3}\ar@{-{*}}[r]&\cdots&&\cdots\ar@{-{*}}[r]&z_i\ar@{-{*}}[r]&z''_{i+1}\ar@{-{*}}[r]&z'_{i+2}\ar@{-{*}}[r]&z_{i+3}\ar@{-{*}}[r]&\cdots\\
      \cdots\ar@{-{*}}[r]&z_i\ar@{-{*}}[r]&z''_{i+1}\ar@{-{*}}[r]&z_{i+2}\ar@{-{*}}[r]&z_{i+3}\ar@{-{*}}[r]&\cdots&&\cdots\ar@{-{*}}[r]&z_i\ar@{-{*}}[r]&z''_{i+1}\ar@{-{*}}[r]&z_{i+2}\ar@{-{*}}[r]&z_{i+3}\ar@{-{*}}[r]&\cdots\\
      &&&\vdots&&&&&&\vdots&&&&.
    }
    \]
    (This corresponds to the six faces of the cube $\vcenter{\xymatrix@R=0.8ex@C=1.5ex@M=0.7ex{&&z_{i+1}\ar@{-{*}}[r]\ar@{-{*}}[dr]&z_{i+2}''\ar@{-{*}}[dr]\\\cdots\ar@{-{*}}[r]&z_{i}\ar@{-{*}}[ur]\ar@{-{*}}[r]\ar@{-{*}}[dr]&z_{i+1}'\ar@{-{*}}[ur]\ar@{-{*}}[dr]&z_{i+2}'\ar@{-{*}}[r]&z_{i+3}\ar@{-{*}}[r]&\cdots\\&&z_{i+1}''\ar@{-{*}}[r]\ar@{-{*}}[ur]&z_{i+2}\ar@{-{*}}[ur]}}$.)
  \end{enumerate}
  \end{enumerate}
  These facts are easy to check directly. Alternatively, they are obvious from a geometric reformulation. The maximal chains in the cube correspond to the vertices of the
  permutohedron $\Permu{k-1}$. The edges of $\Permu{k-1}$ correspond
  to swaps across two-dimensional faces, as described in
  \ref{item:max-chain:2d}, and the two-dimensional faces of
  $\Permu{k-1}$ are either squares or hexagons, corresponding to the
  second and third moves of \ref{item:max-chain:3d}. (For more details
  about the permutohedron, see, for instance,~\cite{Zie-top-polytopes}.) Therefore,
  \ref{item:max-chain:2d} can be restated by saying that any two vertices
  of $\Permu{k-1}$ can be connected by a path along the edges, and
  \ref{item:max-chain:3d} can be restated by saying that any two such paths
  can be connected by homotoping across two-dimensional faces.

  Now consider the given data~\ref{data:set}--\ref{data:bijection}
  that satisfies
  conditions~\ref{condition:matching}--\ref{condition:hexagon}. For
  each $u\geq_k v$, choose a maximal chain
  $\mf{m}^{u,v}=\{\xymatrix@C=1.5ex{u=z^{u,v}_0\ar@{-{*}}[r]&\cdots\ar@{-{*}}[r]&z^{u,v}_i\ar@{-{*}}[r]&\cdots\ar@{-{*}}[r]&z^{u,v}_k=v}\}$ and
  define
  \[
  F(\cmorph{u}{v})=F(\cmorph{z^{u,v}_{k-1}}{z^{u,v}_k})\times_{F(z^{u,v}_{k-1})}\cdots\times_{F(z^{u,v}_1)} F(\cmorph{z^{u,v}_0}{z^{u,v}_1}).
  \]
  For $u\geq_k v\geq_\ell w$, define the $2$-isomorphism
  $F_{u,v,w}\from F(\cmorph{v}{w})\times_{F(v)}F(\cmorph{u}{v})\to
  F(\cmorph{u}{w})$ by choosing a sequence of maximal chains
  connecting $\mf{m}^{v,w}\cup\mf{m}^{u,v}$ to $\mf{m}^{u,w}$ of the
  type described in \ref{item:max-chain:2d}, and using the maps
  provided by the data~\ref{data:bijection}. Since the data satisfies
  conditions~\ref{condition:matching}--\ref{condition:hexagon},
  \ref{item:max-chain:3d} implies that the $2$-isomorphism is
  independent of the choice of the sequence of maximal chains. The
  same argument shows that these $2$-isomorphisms satisfy
  \Diagram{lax-functor-pentagon}. Therefore, this defines a strictly
  unitary lax $2$-functor $F\from\CCat{n}\to\BurnsideCat$. The
  construction is clearly unique up to natural isomorphism.
\end{proof}

Notice that the existence of the $F_{u,v,v',w}$ implies the following:
\begin{enumerate}[label=(C-\arabic*),ref=(C-\arabic*),start=0]
\item\label{condition:bijective} for every two-dimensional face
  $\vcenter{\xymatrix@R=0ex@C=1ex{&v\ar@{-{*}}[dr]\\u\ar@{-{*}}[ur]\ar@{-{*}}[dr]&&w\\&v'\ar@{-{*}}[ur]}}$,
  \(\card{F(\cmorph{v}{w})\times_{F(v)}
    F(\cmorph{u}{v})}=\card{F(\cmorph{v'}{w})\times_{F(v')}
    F(\cmorph{u}{v'})}.\)
\end{enumerate}
The converse is not true: given
data~\ref{data:set}--\ref{data:correspondence} satisfying
condition~\ref{condition:bijective}, there might be no way, one way,
or more than one way of constructing data~\ref{data:bijection}
satisfying
conditions~\ref{condition:matching}--\ref{condition:hexagon}, as the
following two examples illustrate.

\begin{example}\label{exam:multiple-extend}
  Assume we have the following
  data~\ref{data:set}--\ref{data:correspondence} on the $2$-dimensional
  cube
  $\vcenter{\xymatrix@R=0ex@C=1ex{&10\ar@{-{*}}[dr]\\11\ar@{-{*}}[ur]\ar@{-{*}}[dr]&&00\\&01\ar@{-{*}}[ur]}}$:
  \[\xymatrix{
    &\circ\ar@{-}@/^/[dr]|-{b_1}\ar@{-}@/_/[dr]|-{b_2}\\
    \circ\ar@{-}@/^/[ur]|-{a_1}\ar@{-}@/_/[ur]|-{a_2}\ar@{-}@/^/[dr]|-{c_1}\ar@{-}@/_/[dr]|-{c_2}&&\circ\\
    &\circ\ar@{-}@/^/[ur]|-{d_1}\ar@{-}@/_/[ur]|-{d_2}
  }\]
  (We have not labeled the elements of the one-element sets $F(v)$,
  only the elements of the two-element correspondences
  $F(\cmorph{u}{v})$.)  This does not uniquely specify
  data~\ref{data:bijection}. Up to natural isomorphism, we may assume
  \[
  F_{11,10,01,00}\from\{a_1b_1,a_1b_2,a_2b_1,a_2b_2\}\to\{c_1d_1,c_1d_2,c_2d_1,c_2d_2\}
  \]
  sends $a_1b_1$ to $c_1d_1$; but there are still six ways to
  define the bijection $F_{11,10,01,00}$ that are not naturally
  isomorphic. (Indeed, some of these six functors are not even
  equivalent in the sense of \Definition{equiv-functors}; see
  \Remark{above-two-diff}.)
\end{example}

\begin{example}\label{exam:zero-extend}
  Assume we have the following
  data~\ref{data:set}--\ref{data:correspondence} on the
  $3$-dimensional cube
  $\vcenter{\xymatrix@R=0ex@C=0ex{
      111\ar@{-{*}}[rr]\ar@{-{*}}[dr]\ar@{-{*}}[dd]&&110\ar@{-{*}}[dr]\ar@{-{*}}[dd]\\
      &101\ar@{-{*}}[rr]\ar@{-{*}}[dd]&&100\ar@{-{*}}[dd]\\
      011\ar@{-{*}}[dr]\ar@{-{*}}[rr]&&010\ar@{-{*}}[dr]\\
      &001\ar@{-{*}}[rr]&&000
    }}$:
  \[\xymatrix@C=0.1ex@R=0.1ex@M=0ex{
    & && &&& && &&&\circ\ar@{-}[9,3]|(0.3){b_1}\ar@{-}[16,0]\ar@{-}[18,-2]\\
    &p_1\ar@{-}[8,4]|(0.7){c_2}\ar@{-}[7,5]|(0.7){c_1}\ar@{-}[0,9]|(0.7){a_2}\ar@{-}[-1,10]|(0.7){a_1}\ar@{-}[16,0]\ar@{-}[17,-1] && &&& && &&\circ\ar@{-}[8,4]|(0.3){b_2}\ar@{-}[16,0]\ar@{-}[18,-2]&\\
    p_2\ar@{-}[8,4]|(0.7){c_3}\ar@{-}[9,3]|(0.7){c_4}\ar@{-}[1,8]|(0.7){a_4}\ar@{-}[0,9]|(0.7){a_3}\ar@{-}[16,0]\ar@{-}[15,1]& && &&& && &\circ\ar@{-}[8,4]|(0.4){b_3}\ar@{-}[16,0]\ar@{-}[14,2] &&\\
    & && &&& && \circ\ar@{-}[7,5]|(0.5){b_4}\ar@{-}[16,0]\ar@{-}[14,2]&&&\\
    \txt{\phantom{X}}& &\txt{\phantom{XXXX}}& &&&
    &\txt{\phantom{XXXX}}& &&& &\txt{\phantom{XXXX}}&\\
    \txt{\phantom{X}}\\
    \txt{\phantom{X}}\\
    \txt{\phantom{X}}\\
    & && &&&\circ\ar@{-}[1,8]|(0.3){d_1}\ar@{-}[16,0]\ar@{-}[18,-2]\\
    & && &&\circ\ar@{-}[0,9]|(0.3){d_2}\ar@{-}[16,0]\ar@{-}[18,-2]& && &&& && &\circ\ar@{-}[16,0]\ar@{-}[17,-1]\\
    & && &\circ\ar@{-}[0,9]|(0.3){d_3}\ar@{-}[16,0]\ar@{-}[15,1]&& && &&& && \circ\ar@{-}[16,0]\ar@{-}[15,1]&\\
    & && \circ\ar@{-}[-1,10]|(0.3){d_4}\ar@{-}[16,0]\ar@{-}[13,3]&&& \\
    \txt{\phantom{X}}\\
    \txt{\phantom{X}}\\
    \txt{\phantom{X}}\\
    \txt{\phantom{X}}\\
    & && &&& && &&&\circ\ar@{-}[9,3]|(0.3){f_1}\\
    &\circ\ar@{-}[8,4]|(0.6){g_2}\ar@{-}[7,5]|(0.5){g_1}\ar@{-}[0,9]|(0.7){e_2}\ar@{-}[-1,10]|(0.7){e_1} && &&& && &&\circ\ar@{-}[8,4]|(0.3){f_2}&\\
    \circ\ar@{-}[8,4]|(0.7){g_3}\ar@{-}[9,3]|(0.7){g_4}\ar@{-}[1,8]|(0.7){e_4}\ar@{-}[0,9]|(0.7){e_3}& && &&& && &\circ\ar@{-}[8,4]|(0.3){f_3} &&\\
    & && &&& && \circ\ar@{-}[7,5]|(0.3){f_4}&&&\\
    \txt{\phantom{X}}\\
    \txt{\phantom{X}}\\
    \txt{\phantom{X}}\\
    \txt{\phantom{X}}\\
    & && &&&\circ\ar@{-}[1,8]|(0.3){h_1}\\
    & && &&\circ\ar@{-}[0,9]|(0.3){h_2}& && &&& && &q_1\\
    & && &\circ\ar@{-}[0,9]|(0.3){h_3}&& && &&& && q_2&\\
    & && \circ\ar@{-}[-1,10]|(0.3){h_4}&&&
  }
  \]
  (We have only labeled the elements of the two-element sets $F(111)$
  and $F(000)$, and the elements of some of the correspondences
  $F(\cmorph{u}{v})$.
In case the picture is not clear, the
  correspondence $F(\cmorph{110}{010})$ has eight elements with the
  following (source, target) pairs: $(t(a_1),t(e_1))$,
  $(t(a_1),t(e_3))$, $(t(a_2),t(e_2))$, $(t(a_2),t(e_4))$,
  $(t(a_3),t(e_1))$, $(t(a_3),t(e_3))$, $(t(a_4),t(e_2))$, and
  $(t(a_4),t(e_4))$; and the correspondence $F(\cmorph{101}{001})$
  also has eight elements with the following (source, target) pairs:
  $(t(c_1),t(g_1))$, $(t(c_1),t(g_3))$, $(t(c_2),t(g_2))$,
  $(t(c_2),t(g_4))$, $(t(c_3),t(g_2))$, $(t(c_3),t(g_3))$,
  $(t(c_4),t(g_1))$, and $(t(c_4),t(g_4))$.)
  Let us attempt
  to construct data~\ref{data:bijection}. Up to natural isomorphism,
  we may assume the bijection $F_{111,110,101,100}$ sends $a_1b_1$ to
  $c_1d_1$. Looking at the correspondence from $p_1$ to $q_1$,
  conditions~\ref{condition:matching}--\ref{condition:hexagon} force
  $F_{011,010,001,000}(e_1f_1)=g_1h_1$. Similarly, using the
  correspondence from $p_1$ to $q_2$, we get
  $F_{011,010,001,000}(e_3f_3)=g_3h_3$. Using the correspondence from
  $p_2$ to $q_1$, the former implies
  $F_{111,110,101,100}(a_3b_3)=c_4d_4$; but using the correspondence
  from $p_2$ to $q_2$, the latter implies
  $F_{111,110,101,100}(a_3b_3)=c_3d_3$, which is a contradiction.  So
  even though the given data satisfies
  condition~\ref{condition:bijective}, it is impossible to construct
  data~\ref{data:bijection} satisfying
  conditions~\ref{condition:matching}--\ref{condition:hexagon}.
\end{example}

\section{Properties of such functors}\label{sec:equivalence}

In this section, we will discuss a few constructions that can be
carried out with stable functors from the cube to the
Burnside category.

\begin{definition}\label{def:totalization}
  If $F\from\CCat{n}\to\BurnsideCat$, then $\Total(F)\in\Complexes$ is
  defined to be the following chain complex. The chain group is
  defined to be
  \[
  \Total(F)=\bigoplus_{v\in\CCat{n}}\AbFunc(F(v)),
  \]
  the homological grading of the summand $\AbFunc(F(v))$ is defined to be
  $\card{v}$, and the differential $\del\from\Total(F)\to\Total(F)$ is
  defined by declaring the component $\del_{u,v}$ of $\del$ that maps
  from $\AbFunc(F(u))$ to $\AbFunc(F(v))$ to be
  \[
  \del_{u,v}=\begin{cases}(-1)^{s_{u,v}}\AbFunc(F(\cmorph{u}{v}))&\text{if
      $u\geq_1 v$,}\\0&\text{otherwise.}\end{cases}
  \]
  (Here $s_{u,v}$ is the sign assignment from \Definition{sign-assign}
  and $\AbFunc\from\BurnsideCat\to\AbelianGroups$ is the functor from
  \Definition{abfunc}.)  For all $r\in\Z$, define $\Total(\Sigma^r
  F)=\Sigma^r\Total(F)$, the chain complex with gradings shifted up by
  $r$.
\end{definition}

\begin{remark}
  In terms of the reformulation from \Proposition{redefine}, the
  functor $F\from\CCat{n}\to\BurnsideCat$ is equivalent to
  data~\ref{data:set}--\ref{data:bijection} satisfying
  conditions~\ref{condition:matching}--\ref{condition:hexagon}. In
  order to define the chain complex $\Total(F)$, it is enough to have
  data~\ref{data:set}--\ref{data:correspondence} satisfying
  condition~\ref{condition:bijective}.
\end{remark}

\begin{definition}
  If $F,F'\from\CCat{n}\to\BurnsideCat$ are two $2$-functors, then the
  \emph{coproduct} $F\amalg F'\from\CCat{n}\to\BurnsideCat$ is defined as
  follows.
  \begin{enumerate}[leftmargin=*]
  \item For all $v\in\CCat{n}$, $(F\amalg F')(v)=F(v)\amalg F'(v)$.
  \item For all $u>v$, $(F\amalg
    F')(\cmorph{u}{v})=F(\cmorph{u}{v})\amalg F'(\cmorph{u}{v})$ with
    the source and target maps defined component-wise.
  \item For all $u>v>w$, the map $(F\amalg F')_{u,v,w}$ from $(F\amalg
    F')(\cmorph{v}{w})\times_{(F\amalg F')(v)}(F\amalg
    F')(\cmorph{u}{v})\cong
    \big(F(\cmorph{v}{w})\times_{F(v)}F(\cmorph{u}{v})\big)\amalg\big(F'(\cmorph{v}{w})\times_{F'(v)}F'(\cmorph{u}{v})\big)$
    to $(F\amalg F')(\cmorph{u}{w})=F(\cmorph{u}{w})\amalg
    F'(\cmorph{u}{w})$ is defined to be $F_{u,v,w}$ on the first
    component and $F'_{u,v,w}$ on the second component.
  \end{enumerate}
  It is straightforward to check that this defines a strictly unitary
  lax $2$-functor $\CCat{n}\to\BurnsideCat$, and $\Total(F\amalg
  F')=\Total(F)\oplus\Total(F')$.
\end{definition}

\begin{definition}
  If $F_1\from\CCat{n_1}\to\BurnsideCat$, $F_2\from\CCat{n_2}\to\BurnsideCat$
  are two $2$-functors, then the \emph{product} $F_1\times
  F_2\from\CCat{n_1+n_2}\to\BurnsideCat$ is defined as follows.
  \begin{enumerate}[leftmargin=*]
  \item For all $(v_1,v_2)\in\CCat{n_1}\times\CCat{n_2}$, $(F_1\times
    F_2)((v_1,v_2))=F_1(v_1)\times F_2(v_2)$.
  \item For all $(u_1,u_2)>(v_1,v_2)$, $(F_1\times
    F_2)(\cmorph{(u_1,u_2)}{(v_1,v_2)})=F_1(\cmorph{u_1}{v_1})\times
    F_2(\cmorph{u_2}{v_2})$ with the source and target maps defined
    component-wise, with the understanding that if $u_i=v_i$ the
    correspondence $F_i(\cmorph{u_i}{v_i})$ is the identity.
  \item For all $(u_1,u_2)>(v_1,v_2)>(w_1,w_2)$, the map $(F_1\times
    F_2)_{(u_1,u_2),(v_1,v_2),(w_1,w_2)}$ is defined as
    \[
    (F_1\times
    F_2)_{(u_1,u_2),(v_1,v_2),(w_1,w_2)}(x_1,x_2)=\big((F_{1})_{u_1,v_1,w_1}(x_1),
    (F_{2})_{u_2,v_2,w_2}(x_2)\big),
    \]
    with the understanding that if $u_i=v_i$ or $v_i=w_i$ then
    $(F_i)_{u_i,v_i,w_i}$ is the identity map.
\end{enumerate}
Once again, it is straightforward to check that this defines a
strictly unitary lax $2$-functor $\CCat{n_1+n_2}\to\BurnsideCat$. This time,
$\Total(F_1\times F_2)=\Total(F_1)\otimes\Total(F_2)$: the
sign assignment from \Definition{sign-assign} translates into the
Koszul sign convention on the tensor product.
\end{definition}

\begin{definition}
  A \emph{face inclusion} $\iota$ is a functor $\CCat{n}\to\CCat{N}$
  that is injective on objects, and preserves the relative
  gradings. Face inclusions can be described as functors of the
  following form: Fix $U\geq_n W$ in $\CCat{N}$, and let
  $\{V_1,\dots,V_n\}=\Set{V\in\CCat{N}}{U\geq_{n-1}V\geq_1 W}$. The
  functor that sends the object $(v_1,\dots,v_n)\in\CCat{n}$ to the
  object $(W+\sum_i v_i (V_i-W))\in\CCat{N}$ is a face inclusion.

  Let $\card{\iota}=\card{\iota(v)}-\card{v}$ for any $v\in\CCat{n}$
  be the grading of $W$.
\end{definition}

\begin{remark}
  The autoequivalences $\iota\from\CCat{n}\to\CCat{n}$ are face
  inclusions. Note that the group of autoequivalences of $\CCat{n}$ is the
  permutation group $\permu{n}$, where $\sigma\in\permu{n}$
  corresponds to the autoequivalence that sends
  $(v_1,\dots,v_n)\in\CCat{n}$ to
  $(v_{\si^{-1}(1)},\dots,v_{\si^{-1}(n)})$.
\end{remark}

\begin{definition}
  If $\iota\from\CCat{n}\into\CCat{N}$ is a face inclusion, and
  $F\from\CCat{n}\to\BurnsideCat$ is a functor, then the induced
  functor $F_{\iota}\from\CCat{N}\to\BurnsideCat$ is uniquely defined
  by imposing $F=F_{\iota}\circ \iota$, and for all objects
  $v\in\CCat{N}\setminus\iota(\CCat{n})$, $F_{\iota}(v)=\emptyset$.

  Note that the chain complexes $\Total(F_{\iota})$ and
  $\Sigma^{\card{\iota}}\Total(F)$ are canonically isomorphic except
  for signs. For all $u\geq_1 v$ in $\CCat{n}$, the component of the
  differential from $\AbFunc(F(u))$ to $\AbFunc(F(v))$ has sign
  $(-1)^{s_{\iota(u),\iota(v)}}$ in $\Total(F_{\iota})$ and sign
  $(-1)^{\card{\iota}+s_{u,v}}$ in
  $\Sigma^{\card{\iota}}\Total(F)$. To every $v\in\CCat{n}$, assign
  $t_v\in\Z/2$, so that for all $u\geq_1 v$,
  $t_u+t_v=\card{\iota}+s_{u,v}+s_{\iota(u),\iota(v)}$. (Such an
  assignment exists, and is unique up to adding a constant to all the
  assignments.) Then the map
  $\Total(F_{\iota})\to\Sigma^{\card{\iota}}\Total(F)$, defined to be
  $(-1)^{t_v}$ times the identity on the summand $\AbFunc(F(v))$ for
  all $v$ is an isomorphism of chain complexes, and is canonical up
  to an overall sign. 
\end{definition}

\begin{definition}\label{def:natural-transform}
  A \emph{natural transformation} $\eta\from F\to F'$ between two
  $2$-functors $F,
  F'\from\CCat{n}\to\BurnsideCat$ is a strictly unitary lax
  $2$-functor $\eta\from\CCat{n+1}\to\BurnsideCat$ so that
  $\eta|_{\{1\}\times\CCat{n}}=F$ and
  $\eta|_{\{0\}\times\CCat{n}}=F'$ (with respect to the obvious
  identifications of $\{i\}\times \CCat{n}$ with $\CCat{n}$). 

  Note that a natural transformation $\eta$ induces a chain map
  $\Total(\eta)\from\Total(F)\to\Total(F')$, and this is functorial in
  the following sense: If $\eta\from F\to F'$ and $\theta\from F'\to
  F''$ are two natural transformations, then
  $\Total(\theta\circ\eta)=\Total(\theta)\circ\Total(\eta)$. Moreover,
  $\Total$ of the functor $\CCat{n+1} \to \BurnsideCat$
  determined by $\eta$ is the mapping cone of the chain map
  $\Total(\eta)$.
\end{definition}

\begin{definition}\label{def:equiv-functors}
  Two stable functors $(E_1\from\CCat{m_1}\to\BurnsideCat,q_1)$ and
  $(E_2\from\CCat{m_2}\to\BurnsideCat,q_2)$ are defined to be
  \emph{stably equivalent} if there is a sequence of stable functors
  $\{(F_i\from\CCat{n_i}\to\BurnsideCat,r_i)\}$ for $0\leq i\leq
  \ell$, with $\Sigma^{q_1}E_1=\Sigma^{r_0}F_0$ and
  $\Sigma^{q_2}E_2=\Sigma^{r_\ell}F_{\ell}$, such that for every
  adjacent pair $\{\Sigma^{r_i}F_i,\Sigma^{r_{i+1}}F_{i+1}\}$, one of
  the following holds:
  \begin{enumerate}[leftmargin=*]
  \item $(n_i,r_i)=(n_{i+1},r_{i+1})$, and there is a natural
    transformation $\eta$, either from $F_i$ to $F_{i+1}$ or from
    $F_{i+1}$ to $F_i$, so that the induced map $\Total(\eta)$ is a chain homotopy
    equivalence.
  \item $r_{i+1}=r_i-\card{\iota}$ and $F_{i+1}=(F_i)_{\iota}$ for
    some face inclusion $\iota\from\CCat{n_i}\into\CCat{n_{i+1}}$; or
    $r_{i+1}=r_i+\card{\iota}$ and $F_i=(F_{i+1})_{\iota}$ for some face inclusion
    $\iota\from\CCat{n_{i+1}}\into\CCat{n_i}$.
  \end{enumerate}
  We call such a sequence $\{\Sigma^{r_i}F_i\}$ a \emph{stable
    equivalence} between $\Sigma^{q_1}E_1$ and $\Sigma^{q_2}E_2$. Note
  that a stable equivalence induces a chain homotopy equivalence
  $\Total(\Sigma^{q_1}E_1)\to\Total(\Sigma^{q_2}E_2)$, well-defined up
  to choices of inverses of chain homotopy equivalences and an overall
  sign.
\end{definition}

As is standard, instead of having to consider an arbitrary sequence of
zig-zags of natural transformations, it is enough to consider a single
zig-zag.

\begin{proposition}
  If stable functors $(E_1\from\CCat{m_1}\to\BurnsideCat,q_1)$ and
  $(E_2\from\CCat{m_2}\to\BurnsideCat,q_2)$ are stably equivalent,
  then there exist stable functors
  $(F_1\from\CCat{n}\to\BurnsideCat,r)$,
  $(F_2\from\CCat{n}\to\BurnsideCat,r)$, and
  $(G\from\CCat{n}\to\BurnsideCat,r)$, satisfying the following for
  all $i\in\{1,2\}$:
  \begin{enumerate}[leftmargin=*]
  \item $F_i=(E_{i})_{\iota_i}$ for some face inclusion
    $\iota_i\from\CCat{m_i}\into\CCat{n}$, and $q_i=r+\card{\iota_i}$.
  \item There is a natural transformation $\eta_i\from F_i\to G$, so
    that $\Total(\eta_i)$ is a chain homotopy equivalence.
  \end{enumerate}
\end{proposition}

\begin{proof}
  We can compose natural transformations $\eta\from F\to F'$ and
  $\eta'\from F'\to F''$ to get a natural transformation $\eta'\circ
  \eta\from F\to F''$; and since
  $\Total(\eta'\circ\eta)=\Total(\eta')\circ\Total(\eta)$, if
  $\Total(\eta)$ and $\Total(\eta')$ are chain homotopy equivalences,
  so is $\Total(\eta'\circ\eta)$. Similarly, for any face inclusions
  $\iota\from\CCat{n}\into \CCat{n'}$ and
  $\iota'\from\CCat{n'}\into\CCat{n''}$, the composition
  $\iota'\circ\iota\from\CCat{n}\into\CCat{n''}$ is a face inclusion
  with $\card{\iota'\circ\iota}=\card{\iota}+\card{\iota'}$; and for
  any $F\from\CCat{n}\to\BurnsideCat$,
  $(F_{\iota})_{\iota'}=F_{\iota'\circ\iota}$. Finally, if $\eta\from
  F\to F'$ is a natural transformation between functors
  $F,F'\from\CCat{n}\to\BurnsideCat$, and
  $\iota\from\CCat{n}\into\CCat{N}$ is a face inclusion, then there is
  an induced natural transformation $\eta_{\iota}\from F_{\iota}\to
  F'_{\iota}$; and if $\Total(\eta)$ is a chain homotopy equivalence,
  so is $\Total(\eta_{\iota})$.

  Using these moves, we can convert any stable equivalence
  $\{\Sigma^{r_i}F_i\}$ from $\Sigma^{q_1}E_1$ to $\Sigma^{q_2}E_2$ to
  one of the form $\{\Sigma^{q_1}E_1,\Sigma^{r}F_1=\Sigma^r
  G_0,\Sigma^{r}G_1,\dots,
  \Sigma^{r}G_{\ell-1},\Sigma^{r}G_\ell=\Sigma^{r}F_2,
  \Sigma^{q_2}E_2\}$, where $G_0,\dots,G_{\ell}$ are all functors
  $\CCat{n}\to\BurnsideCat$; $F_i=(E_i)_{\iota_i}$ for some face
  inclusion $\iota_i\from\CCat{m_i}\into\CCat{n}$ and $q_i=r+\card{\iota_i}$; and
  there is a zig-zag of natural transformations connecting
  $\{G_0,\dots,G_\ell\}$, inducing chain homotopy equivalences among
  $\Total(G_i)$.

  So, in order to prove the proposition, it is enough to show that we
  can we can convert a zig-zag of the
  form $F\stackrel{\eta}{\leftarrow} G\stackrel{\eta'}{\rightarrow}
  F'$ with $\Total(\eta)$ and $\Total(\eta')$ chain homotopy
  equivalences to a zig-zag of the form $F\stackrel{\theta}{\rightarrow}
  H\stackrel{\theta'}{\leftarrow} F'$ with $\Total(\theta)$ and
  $\Total(\theta')$ chain homotopy equivalences. 
  We can achieve this
  by working on the cube $\CCat{n+1}$ instead of $\CCat{n}$. That is,
  we will construct $H\from\CCat{n+1}\to\BurnsideCat$, and
  $\theta\from F_{\iota_0}\to H$, $\theta'\from F'_{\iota_0}\to H$ so
  that $\Total(\theta)$ and $\Total(\theta')$ are chain homotopy
  equivalences (with $\iota_0$ denoting the face inclusion
  $\CCat{n}\cong\{0\}\times\CCat{n}\into\CCat{n+1}$).

  To define $H$, note that the natural transformations $\eta$ and
  $\eta'$ are thought of as functors $\CCat{n+1}\to\BurnsideCat$
  satisfying $\eta|_{\{1\}\times\CCat{n}}=
  \eta'|_{\{1\}\times\CCat{n}}=G_{\iota_1}|_{\{1\}\times\CCat{n}}$
  (with $\iota_1$ denoting the face inclusion
  $\CCat{n}\cong\{1\}\times\CCat{n}\into\CCat{n+1}$), and
  $\eta|_{\{0\}\times\CCat{n}}=F_{\iota_0}|_{\{0\}\times\CCat{n}}$,
  and
  $\eta'|_{\{0\}\times\CCat{n}}=F'_{\iota_0}|_{\{0\}\times\CCat{n}}$. Define
  $H$ to be the quotient of $\eta\amalg\eta'$ by identifying
  $\eta|_{\{1\}\times\CCat{n}}$ with
  $\eta'|_{\{1\}\times\CCat{n}}$. The natural transformations
  $F_{\iota_0}\stackrel{\theta}{\to} H$ and
  $F'_{\iota_0}\stackrel{\theta'}{\to} H$ come from inclusions. Since
  $\Total(\eta)$ is a chain homotopy equivalence, so is
  $\Total(\theta')$; and since $\Total(\eta')$ is a chain homotopy
  equivalence, so is $\Total(\theta)$. This concludes the proof.
\end{proof}

We conclude this section with some illustrative examples.

\begin{example}\label{exam:rp2-smash-rp2}
  Let $\mb{P}\from\CCat{}\to\BurnsideCat$ denote the functor that
  assigns one-element sets to $1$ and $0$, and a two-element
  correspondence to $\cmorph{1}{0}$. Extend the data from
  \Example{multiple-extend} to a functor
  $F^{(1)}\from\CCat{2}\to\BurnsideCat$ by declaring the matching
  $F^{(1)}_{11,10,01,00}$ to be the map
  \[
  a_1b_1\mapsto c_1d_1\qquad a_1b_2\mapsto c_2d_1\qquad a_2b_1\mapsto c_1d_2\qquad a_2b_2\mapsto c_2d_2.
  \]
  Then $F^{(1)}$ is naturally isomorphic to $\mb{P}\times\mb{P}$.
\end{example}

\begin{example}\label{exam:rp2-wedge-rp2}
  Consider once again the data from \Example{multiple-extend}; this
  time define the functor $F^{(2)}\from\CCat{2}\to\BurnsideCat$ by
  declaring the matching $F^{(2)}_{11,10,01,00}$ to be the map
  \[
  a_1b_1\mapsto c_1d_1\qquad a_1b_2\mapsto c_1d_2\qquad a_2b_1\mapsto c_2d_1\qquad a_2b_2\mapsto c_2d_2.
  \]
  Then $F^{(2)}$ is stably equivalent to
  $\mb{P}_{\iota_0}\amalg\mb{P}_{\iota_1}$, where $\iota_i$ is the
  face inclusion $\CCat{}\cong\{i\}\times\CCat{}\into\CCat{2}$. To see
  this, let $G$ be the following diagram on the cube
  $\vcenter{\xymatrix@R=0ex@C=0ex{
      111\ar@{-{*}}[rr]\ar@{-{*}}[dr]\ar@{-{*}}[dd]&& 110\ar@{-{*}}[dr] \ar@{-{*}}[dd]\\
      &101\ar@{-{*}}[rr]\ar@{-{*}}[dd]&&100\ar@{-{*}}[dd]\\
      011\ar@{-{*}}[dr]\ar@{-{*}}[rr]&&010\ar@{-{*}}[dr]\\
      &001\ar@{-{*}}[rr]&&000 }}$:
\[
\xymatrix@C=0.1ex@R=0.5ex@M=0ex{
\circ\ar@{-}[rrrrrr]\ar@{-}[dddrrr]|2\ar@{-}[dddddd]\ar@{-}[ddddddr]|(0.7)2& && & && p_3\ar@{-}[dddrrr]|2\ar@{-}[dddddd]|(0.7)2\\
\txt{\phantom{X}}& &\txt{\phantom{XXX}}& \txt{\phantom{X}}& &\txt{\phantom{XXX}}&
\txt{\phantom{X}}& &\txt{\phantom{XXX}}& \txt{\phantom{X}}\\
\txt{\phantom{X}}\\
& && \circ\ar@{-}[rrrrrr]\ar@{-}[dddddd]\ar@{-}[ddddddr]& && & &&p_4\ar@{-}[dddddd]|2\\
\txt{\phantom{X}}\\
\txt{\phantom{X}}\\
p_5\ar@{-}[dddrrr]|2&\circ\ar@{-}[rrrrr]\ar@{-}[dddrrr] && & &&p_1\ar@{-}[dddrrr]|2\\
\txt{\phantom{X}}\\
\txt{\phantom{X}}\\
& && p_6&\circ\ar@{-}[rrrrr]|2 && & &&p_2.
}
\]
(We have only labeled some of the elements of some of the sets
$G(v)$. We have not labeled the elements of the correspondences
$G(\cmorph{u}{v})$, but merely indicated their cardinalities if bigger
than one.) Define the matching $G_{110,100,010,000}$ on the rightmost
face to be isomorphic to the matching $F^{(2)}_{11,10,01,00}$
above. It is straightforward to verify that we can construct matchings
on the other two-dimensional faces in a unique fashion so that
condition~\ref{condition:hexagon} from \Section{functor} is
satisfied. Therefore, this defines a functor
$G\from\CCat{3}\to\BurnsideCat$.

Let $F$ be the diagram restricted to the objects $\{p_1,p_2,p_3,p_4\}$
and the correspondences between them, and let $F'$ be the diagram
restricted to the objects $\{p_1,p_2,p_5,p_6\}$ and the
correspondences between them. It is easy to verify that both the
inclusions $F\to G$ and $F'\to G$ are natural transformations that
induce chain homotopy equivalences $\Total(F)\to \Total(G)$ and
$\Total(F')\to \Total(G)$. Furthermore, $F$ is naturally isomorphic to
$F^{(2)}_{\iota}$ where $\iota$ is the face inclusion
$\CCat{2}\cong\CCat{2}\times\{0\}\into\CCat{3}$; and $F'$ is naturally
isomorphic to $(\mb{P}_{\iota_0}\amalg\mb{P}_{\iota_1})_{\iota'}$
where $\iota'$ is the face inclusion
$\CCat{2}\cong\{0\}\times\CCat{2}\into\CCat{3}$. Therefore, $F^{(2)}$
is stably equivalent to $\mb{P}_{\iota_0}\amalg\mb{P}_{\iota_1}$.
\end{example}

\begin{remark}\label{rem:above-two-diff}
  Properties of the topological realizations
  from \Section{spaces} imply that the two functors
  $F^{(1)}$ and $F^{(2)}$ from
  Examples~\ref{exam:rp2-smash-rp2}--\ref{exam:rp2-wedge-rp2} are not
  stably equivalent. The cohomology of $\Realize{\mb{P}}$ is
  $\F_2$, supported in grading zero (by
  property~\ref{item:property-objects}\ref{item:homolo-ident}), and
  therefore, the spectrum $\Realize{\mb{P}}$ is homotopy equivalent to
  $\Sigma^{-1}\mb{R}\mb{P}^2$. Since $F^{(1)}$ is naturally isomorphic
  to $\mb{P}\times\mb{P}$, the spectrum $\Realize{F^{(1)}}$ is
  homotopy equivalent to
  $\Sigma^{-2}(\mb{R}\mb{P}^2\wedge\mb{R}\mb{P}^2)$ (by
  property~\ref{item:prod-to-prod}). On the other hand, since
  $F^{(2)}$ is stably equivalent to
  $\mb{P}_{\iota_0}\amalg\mb{P}_{\iota_1}$, the spectrum
  $\Realize{F^{(2)}}$ is homotopy equivalent to
  $(\Sigma^{-1}\mb{R}\mb{P}^2)\vee\mb{R}\mb{P}^2$ (by
  properties~\ref{item:coprod-to-coprod},~\ref{item:suspend-equiv},
  and~\ref{item:equiv-equiv}). However, the spectra
  $\Sigma^{-2}(\mb{R}\mb{P}^2\wedge\mb{R}\mb{P}^2)$ and
  $(\Sigma^{-1}\mb{R}\mb{P}^2)\vee\mb{R}\mb{P}^2$ are not homotopy
  equivalent (being distinguished by the Steenrod square $\Sq^2$ for
  instance), and therefore, (by property~\ref{item:equiv-equiv} once
  again) the diagrams $F^{(1)}$ and $F^{(2)}$ are not stably
  equivalent.
\end{remark}

\section{The Khovanov functor}\label{sec:khovanov-functor}
We turn now to the refinement of Khovanov homology. Fix an oriented
link diagram $K$ with $n$ crossings and an ordering of the crossings
of $K$. Khovanov associated an $n$-dimensional cube of free Abelian
groups to this data, as follows~\cite{Kho-kh-categorification}. Each
crossing of $K$ has a $0$-resolution and a $1$-resolution:
\[
\vcenter{\xymatrix{\ar@{-}@/^2.5ex/[d]&\ar@{-}@/_2.5ex/[d]\\&}}\stackrel{0}{\longleftarrow}\vcenter{\xymatrix{\ar@{-}[dr]|-\hole&\ar@{-}[dl]\\&}}\stackrel{1}{\longrightarrow}\vcenter{\xymatrix{\ar@{-}@/_2.5ex/[r]&\\\ar@{-}@/^2.5ex/[r]&}}.
\]
So, given a vertex $v=(v_1,\dots,v_n)\in\{0,1\}^n$ of the cube,
performing the $v_i$-resolution at the $i\th$ crossing gives a
collection $K_v$ of disjoint, embedded circles in $S^2$. Further, for
each edge $\xymatrix@C=1.5ex{u\ar@{-{*}}[r]&v}$ of the cube, $K_u$ is
obtained from $K_v$ by either merging two circles into one or
splitting one circle into two. Let $V=\ZZ\langle x_+, x_- \rangle$ be
a free, rank-$2$ $\ZZ$-module, which we endow with a multiplication
and comultiplication by
\[
m(x_+\otimes x_+)=x_+ \qquad m(x_+\otimes x_-)=x_- \qquad m(x_-\otimes x_+)=x_- \qquad m(x_-\otimes x_-)=0
\]
\[
\Delta(x_+)=x_+\otimes x_-+x_-\otimes x_+ \qquad\qquad \Delta(x_-)=x_-\otimes x_-.
\]
Define a functor $\KhAbFunc\co (\CCat{n})^\op\to
\AbelianGroups$ on objects by setting $\KhAbFunc(v)=\bigotimes_{S\in
  \pi_0(K_v)}V$. On morphisms, if $\xymatrix@C=1.5ex{u\ar@{-{*}}[r]&v}$ is an edge
of the cube so that $K_u$ is obtained from $K_v$ by merging two
circles then $\KhAbFunc(\cmorph{u}{v}^\op)$ applies the multiplication
map to the corresponding factors of $\KhAbFunc(v)$ and the identity
map to the remaining factors; if instead $K_u$ is obtained from $K_v$
by splitting one circle then $\KhAbFunc(\cmorph{u}{v}^\op)$ applies
the comultiplication map to the corresponding factor of $\KhAbFunc(v)$
and the identity map to the remaining factors. It is straightforward
to verify that the resulting diagram commutes. The total complex of
this cube (i.e., multiplying the map on the edge
$\xymatrix@C=1.5ex{u\ar@{-{*}}[r]&v}$ by $(-1)^{s_{u,v}}$ and summing
over vertices of each grading, in a fashion similar to
\Definition{totalization}) is the Khovanov complex $\KhCx(K)$.

The Khovanov cube decomposes as a direct sum over \emph{quantum gradings}. To
define the quantum grading, notice that for $v\in\{0,1\}^n$, $\KhAbFunc(v)$ comes with
a preferred basis: the labelings of the circles in $K_v$ by elements
of $\{x_-,x_+\}$, i.e., functions $\pi_0(K_v)\to \{x_+,x_-\}$.  The
quantum grading of a basis element is
\[
\gr_q(v,x\co \pi_0(K_v)\to \{x_+,x_-\})=n_+-2n_-+\card{v}+\card{\Set{Z}{x(Z)=x_+}}-\card{\Set{Z}{x(Z)=x_-}}\in\ZZ,
\]
where $n_+$ and $n_-$ are the number of positive and negative
crossings of $K$, respectively. There is also a homological
grading shift of $-n_-$.

Our main goal is to refine $\KhAbFunc$ to a functor
$\KhFunc\from\CCat{n}\to\BurnsideCat$, satisfying the following:
\begin{enumerate}[leftmargin=*,label=(\alph*),ref=(\alph*)]
\item For all $v$, $\KhFunc(v)=\{x\co \pi_0(K_v)\to
  \{x_+,x_-\}\}$ is the preferred basis of Khovanov generators.
\item The above identification induces an isomorphism $\Sigma^{-n_-}\Total(\KhFunc)\cong\Dual{\KhCx}.$
\end{enumerate}
By \Proposition{redefine} it suffices to define $\KhFunc$ on vertices,
edges, and $2$-dimensional faces. For all vertices $v$, $F_{\Kh}(v)$
is already defined. Next, notice that for each edge
$\xymatrix@C=1.5ex{u\ar@{-{*}}[r]&v}$ and each element $y\in
F_{\Kh}(v)$, $\KhAbFunc(\cmorph{u}{v}^\op)(y)=\sum_{x\in
  F_{\Kh}(u)}\epsilon_{x,y}\, x$ where $\epsilon_{x,y}\in\{0, 1\}$. (In
other words, all of the entries of the matrix
$\KhAbFunc(\cmorph{u}{v}^\op)$ are $0$ or $1$.) Define
$F_\Kh(\cmorph{u}{v})=\Set{(y,x)\in F_{\Kh}(v)\times
  F_{\Kh}(u)}{\epsilon_{x,y}=1}$, with the obvious source and target
maps to $F_{\Kh}(u)$ and $F_{\Kh}(v)$.

So far, there is no information in $F_{\Kh}$ beyond that in the
Khovanov complex. The new information is in the definition of
$F_{u,v,v',w}\co F_{\Kh}(\cmorph{v}{w})\times_{F_{\Kh}(v)}
F_{\Kh}(\cmorph{u}{v})\to F_{\Kh}(\cmorph{v'}{w})\times_{F(v')}
F_{\Kh}(\cmorph{u}{v'})$ for the $2$-dimensional faces
$\vcenter{\xymatrix@R=0ex@C=1ex{&v\ar@{-{*}}[dr]\\u\ar@{-{*}}[ur]\ar@{-{*}}[dr]&&w\\&v'\ar@{-{*}}[ur]}}$.
The fact that $\KhAbFunc$ is a commutative diagram implies that there
is a $2$-isomorphism between
$F_{\Kh}(\cmorph{v}{w})\times_{F_{\Kh}(v)} F_{\Kh}(\cmorph{u}{v})$ and
$F_{\Kh}(\cmorph{v'}{w})\times_{F_{\Kh}(v')}
F_{\Kh}(\cmorph{u}{v'})$. Namely,
for $x\in F_{\Kh}(u)$ and $z\in F_{\Kh}(w)$, the cardinalities of
\begin{align*}
A_{x,z}&\coloneqq s^{-1}(x)\cap t^{-1}(z)\subseteq
F_{\Kh}(\cmorph{v}{w})\times_{F_{\Kh}(v)}
F_{\Kh}(\cmorph{u}{v}) \qquad \text{and}\\
A'_{x,z}&\coloneqq s^{-1}(x)\cap t^{-1}(z)\subseteq
F_{\Kh}(\cmorph{v'}{w})\times_{F_{\Kh}(v)}
F_{\Kh}(\cmorph{u}{v'})
\end{align*}
are the $(x,z)$ entries in the matrices
$\KhAbFunc(\cmorph{u}{v}^\op)\circ\KhAbFunc(\cmorph{v}{w}^\op)$ and
$ \KhAbFunc(\cmorph{u}{v'}^\op)\circ\KhAbFunc(\cmorph{v'}{w}^\op)$,
respectively, and these two matrices are the same. Further, in most
instances, these sets have $0$ or $1$ elements, so there is a unique
isomorphism $F_{u,v,v',w}|_{A_{x,z}}\co A_{x,z}\to A'_{x,z}$.

\captionsetup{width=0.9\textwidth}
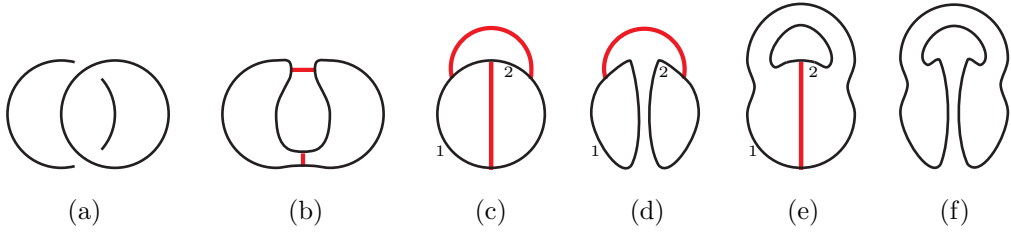
\begin{figure}
  \centering
  \begin{overpic}[tics=5]{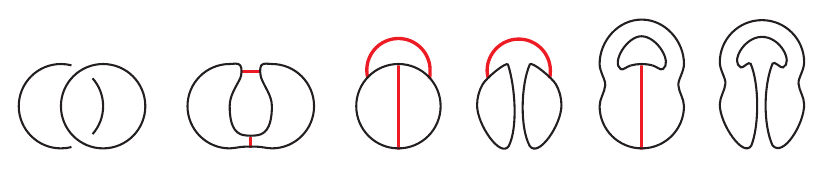}
    \put(43,6){\tiny $1$}
    \put(57.75,6){\tiny $1$}
    \put(72.75,6){\tiny $1$}
    \put(49.5,13.5){\tiny $2$}
    \put(64.25,13.5){\tiny $2$}
    \put(78.5,13.5){\tiny $2$}
    \put(8,0){(a)}
    \put(29,0){(b)}
    \put(46.75,0){(c)}
    \put(61.5,0){(d)}
    \put(76.5,0){(e)}
    \put(91,0){(f)}
  \end{overpic}
  \caption{\textbf{The ladybug matching.} (a) A piece of a (perhaps)
    partially-resolved link diagram. (b) The $00$-resolution, along
    with the arcs $a_v$ and $a_{v'}$, drawn with thick lines. This is
    a ladybug configuration. (c) The same configuration, up to isotopy
    in $S^2$, looking more like a ladybug. The right pair of arcs is
    labeled $1$ and $2$. (d)--(e) The $10$ and $01$ resolutions. The
    circles in the $10$ resolution and the circles in the $01$
    resolution can be identified via the induced numbering by
    $1,2$. (f) The $11$ resolution.}
  \label{fig:ladybug}
\end{figure}

The exceptional case is when there is a circle $C_w$ in $K_w$ which
splits to form two circles in each of $K_v$ and $K_{v'}$, and these
same two circles merge to form a single circle $C_u$ in $K_u$; $x$
labels $C_u$ by $x_-$; and $z$ labels $C_w$ by $x_+$. We call this
configuration a \emph{ladybug configuration} because of the following
depiction. First, draw the circle $C_w$. Then, draw the arc $a_{v}$
with endpoints on $C_w$ which pinches along the edge
$\xymatrix@C=1.5ex{v\ar@{-{*}}[r]&w}$ (i.e., performing embedded
$1$-surgery on $K_w$ along $a_{v}$ produces $K_v$). Similarly, draw
another arc $a_{v'}$ with endpoints on $C_w$ which pinches along the
edge $\xymatrix@C=1.5ex{v'\ar@{-{*}}[r]&w}$. Up to isotopy in $S^2$,
the result looks like the ladybug in Figure~\ref{fig:ladybug}. Now,
distinguish a pair of arcs in $(C_w,\bdy a_v\cup \bdy a_{v'})$, the
\emph{right pair}, as the two arcs you get to by walking along $a_v$
or $a_{v'}$ to $C_w$ and then turning right. (Again, see
Figure~\ref{fig:ladybug}.) Number the two arcs in the right pair as
$1,2$---the numbering will not matter. Label the two circles in $K_v$
which come from $C_w$ as $C_{v}^{1}$ and $C_{v}^{2}$, according to
whether they contain the right arc numbered $1$ or $2$. Similarly,
label the two circles in $K_{v'}$ which come from $C_w$ as $C_{v'}^1$
and $C_{v'}^{2}$, according to whether they contain the right arc
numbered $1$ or $2$. The two elements of $A_{x,z}$ are
\[
  \alpha=((C_w\to x_+), (C_v^{1},C_v^{2})\to (x_-,x_+), (C_u\to x_-)),\ 
  \beta=((C_w\to x_+), (C_v^{1},C_v^{2})\to (x_+,x_-), (C_u\to x_-))
\]
while the two elements of $A'_{x,z}$ are
\[
  \alpha'=((C_w\to x_+), (C_{v'}^{1},C_{v'}^{2})\to (x_-,x_+), (C_u\to x_-)),\ 
  \beta'=((C_w\to x_+), (C_{v'}^{1},C_{v'}^{2})\to (x_+,x_-), (C_u\to x_-))
\]
The bijection $F_{u,v,v',w}$ sends $\alpha$ to $\alpha'$ and $\beta$
to $\beta'$.  (See also~\cite[Section 5.4]{RS-khovanov}.)

\begin{proposition}\label{prop:Kh-burnside-defined}
  The definitions above satisfy
  conditions~\ref{condition:matching}--\ref{condition:hexagon}
  from \Section{functor}, and so induce a strictly unitary $2$-functor
  $\KhFunc\from\CCat{n}\to\BurnsideCat$ by \Proposition{redefine}.
\end{proposition}
\begin{proof}[Sketch of Proof]
  The proof is essentially the same as the proof of~\cite[Proposition
  5.19]{RS-khovanov} and, indeed, the result follows
  from~\cite[Proposition 5.20]{RS-khovanov} and~\cite[Lemma
  4.2]{LLS-khovanov-product}, so we will only sketch the argument
  here.

  Condition~\ref{condition:matching} is immediate, so we only need to
  check condition~\ref{condition:hexagon}.  Fix a three-dimensional
  face
  $\vcenter{\xymatrix@R=0.8ex@C=1.5ex{&v_1\ar@{-{*}}[r]\ar@{-{*}}[dr]&
      v_2''\ar@{-{*}}[dr]\\u\ar@{-{*}}[ur]\ar@{-{*}}[r]
      \ar@{-{*}}[dr]&v_1'\ar@{-{*}}[ur]\ar@{-{*}}[dr]&
      v_2'\ar@{-{*}}[r]&w\\&v_1''\ar@{-{*}}[r]
      \ar@{-{*}}[ur]&v_2\ar@{-{*}}[ur]}}$. We have six correspondences
  from $\KhFunc(u)$ to $\KhFunc(w)$ coming from the six maximal chains
  from $u$ to $w$. Fix Khovanov generators $x\in\KhFunc(u)$ and
  $z\in\KhFunc(w)$, and consider the subsets $s^{-1}(x)\cap t^{-1}(z)$
  of these six correspondences. The various coherence maps
  $F_{u,v_1^{(i)},v_2^{(j)}}$ and $F_{v_1^{(i)},v_2^{(j)},w}$ produce
  six bijections connecting these six sets, and we need to check that
  these bijections agree.

  Unless one of the two-dimensional faces of the above cube is a
  ladybug configuration, each of the above six sets will contain $0$
  or $1$ element, and the check is vacuous or trivial,
  respectively. So we concentrate on configurations that contain
  ladybugs, and check them by hand. Using some underlying symmetries,
  we can reduce to just four configurations
  \cite[Figure~5.3.~a--c,~e]{RS-khovanov}. The first three are
  similar, so there are essentially two case checks that need to be
  performed; and the proofs of \cite[Lemmas~5.14,~5.17]{RS-khovanov}
  imply that the checks succeed.
\end{proof}

Hu-Kriz-Kriz gave an intrinsic description of the functor
$F_{\Kh}$~\cite{HKK-Kh-htpy}, that does not require
\Proposition{redefine}, as follows. On objects, $\HKKfun$ agrees with
$\KhFunc$. To define $\HKKfun$ on morphisms, first fix a checkerboard
coloring for $K$. There is an induced checkerboard coloring of each
resolution $K_v$. Further, each morphism $\cmorph{u}{v}$ corresponds
to a cobordism $\Sigma_{u,v}$ in $[0,1]\times S^2$ from $K_u$ to
$K_v$, and the checkerboard coloring of $K$ induces a coloring of the
components of $([0,1]\times S^2)\setminus \Sigma_{u,v}$. Let
$B_{u,v}\subset ([0,1]\times S^2)\setminus \Sigma_{u,v}$ denote the
closure of the black region. Then 
\begin{equation}\label{eq:HKK-group}
H_1(B_{u,v})/H_1(B_{u,v}\cap\{0,1\}\times S^2)\cong \bigoplus_{\Sigma_0\in \pi_0(\Sigma_{u,v})}\ZZ^{g(\Sigma_0)}.
\end{equation}
If each connected component of $\Sigma_{u,v}$ has genus $0$ or $1$ then we
call a vector $(\epsilon_1,\dots,\epsilon_k)$ in the
group~\eqref{eq:HKK-group} such that each $\epsilon_i=\pm1$ a \emph{sign
  choice} for $\Sigma_{u,v}$. By a \emph{valid boundary labeling} of
$\Sigma_{u,v}$ we mean a labeling of the circles in $\bdy \Sigma_{u,v}$ by $x_+$
or $x_-$ satisfying the condition that every component $\Sigma_0$ of
the cobordism has genus equal to
\[
1-\card{\Set{C\in \pi_0(K_v\cap\bdy \Sigma_0)}{x(C)=x_-}}-\card{\Set{C\in \pi_0(K_u\cap\bdy \Sigma_0)}{x(C)=x_+}}.
\]
(In particular, if some component $\Sigma_0$ has genus bigger than $1$
then the set of valid boundary labelings is empty.) Then
$\HKKfun(\cmorph{u}{v})$ is the product of the set of valid boundary
labelings for $\Sigma_{u,v}$ and the set of sign choices for
$\Sigma_{u,v}$.  The source and target maps for
$\HKKfun(\cmorph{u}{v})$ are given by restricting the boundary
labelings of $\Sigma_{u,v}$ to $K_u$ and $K_v$, respectively. The
coherence map
$\HKKfun(\cmorph{v}{w})\times_{\HKKfun(v)}\HKKfun(\cmorph{u}{v})\to
\HKKfun(\cmorph{u}{w})$ is obvious except when a genus $0$ component
$\Sigma_0$ of $\HKKfun(\cmorph{u}{v})$ and a genus $0$ component
$\Sigma_1$ of $\HKKfun(\cmorph{v}{w})$ glue to give a genus $1$
component $\Sigma$ of $\HKKfun(\cmorph{u}{w})$. In this last case,
there is a single circle $C$ of $\Sigma_0\cap \Sigma_1$, that is
non-separating in $\Sigma$ and is labeled $x_+$. Orient $C$ and
$\Sigma$ as the boundaries of the black regions in $S^2$ and
$[0,1]\times S^2$, respectively, and choose a circle $D$ in
$\Sigma_{u,w}$ with intersection number $D\cdot C=1$. Either the
pushoff of $C$ or the pushoff of $D$ into $B_{u,w}$ is a generator of
the new $\ZZ$-summand of~\eqref{eq:HKK-group}; use this generator to
extend the sign choice to $\Sigma_{u,w}$. This finishes the
construction of $\HKKfun$, and verifying that this does, in fact,
define a 2-functor is straightforward. The functor $\HKKfun$ is
naturally isomorphic to the functor $F_{\Kh}$ defined via the ladybug
matching~\cite[Lemma 8.1]{LLS-khovanov-product}.

\begin{definition}\label{def:khovanov-functor}
  Associated to an oriented $n$-crossing link diagram $K$, let
  $\KhFunc(K)\from\CCat{n}\to\BurnsideCat$ be the functor constructed
  in \Proposition{Kh-burnside-defined}. Define the \emph{Khovanov
    functor} to be the stable functor $\Sigma^{-n_-}\KhFunc(K)$, where
  $n_-$ is the number of negative crossings in $K$.
\end{definition}

There is also a reduced Khovanov functor associated to a pointed link
$(K,p)$: the basepoint $p$ chooses a preferred circle $C_{v,p}$ in
each resolution $K_v$, and we declare that
\begin{align*}
F_{\rKh}(v)&=\{x\in F_{\Kh}(v)\mid x(C_{v,p})=x_-\},\\
F_{\rKh}(\cmorph{u}{v})&=s^{-1}(F_{\rKh}(u))\cap t^{-1}(F_{\rKh}(v))\subseteq F_{\Kh}(\cmorph{u}{v}),
\end{align*}
and similarly the coherence maps $(F_{\rKh})_{u,v,w}$ are restrictions
of the coherence maps $(F_{\Kh})_{u,v,w}$ for $F_{\Kh}$.  It is
straightforward to see that $F_{\rKh}$ does define a strictly unitary
$2$-functor. (Replacing $x_-$ with $x_+$ in the definition would give
a naturally isomorphic functor $F'_{\rKh}$.)
\begin{definition}
  Let $F_{\rKh}(K,p)\from\CCat{n}\to\BurnsideCat$ be the above functor
  associated to a pointed, oriented $n$-crossing link diagram
  $K$. Define the \emph{reduced Khovanov functor} to be the stable
  functor $\Sigma^{-n_-}F_{\rKh}(K)$, where $n_-$ is the number of
  negative crossings in $K$.
\end{definition}

Since the chain complexes $\KhCx$ and $\rKhCx$ decompose according to
quantum gradings, so do the functors $F_{\Kh}$ and $F_{\rKh}$:
\[
F_{\Kh}=\coprod_{j\in\ZZ}F_{\Kh}^j \qquad\qquad F_{\rKh}=\coprod_{j\in\ZZ}F_{\rKh}^j,
\]
where $F_{\Kh}^j(v)=\{x\in F_{\Kh}(v)\mid \gr_q(v,x)=j\}$, and
$F_{\rKh}^j(v)=F_{\rKh}(v)\cap F_{\Kh}^{j-1}(v)$.

\begin{theorem}\label{thm:main}
  If $K_1$ and $K_2$ are oriented link diagrams, with $n_-^1$ and
  $n_-^2$ negative crossings, respectively, representing isotopic
  oriented links then the Khovanov functors
  $\Sigma^{-n^1_-}F_{\Kh}^j(K_1)$ and $\Sigma^{-n^2_-}F_{\Kh}^j(K_2)$
  are stably equivalent stable functors.  Similarly, if $(K_1,p_1)$
  and $(K_2,p_2)$ are pointed, oriented link diagrams, with $n_-^1$ and
  $n_-^2$ negative crossings, respectively, representing isotopic
  pointed, oriented links then the reduced Khovanov functors
  $\Sigma^{-n_-^1}F_{\rKh}^j(K_1,p_1)$ and
  $\Sigma^{-n_-^2}F_{\rKh}^j(K_2,p_2)$ are stably equivalent stable
  functors.
\end{theorem}
\begin{proof}[Sketch of Proof]
  The proof is essentially the same as the proof of invariance of the
  Khovanov spectrum~\cite[Theorem 1.1]{RS-khovanov}, so we only give a
  sketch. We start with a general principle. Since $\KhCx(K)$ comes
  with a preferred basis, so does the dual complex $\Dual{\KhCx(K)}$.
  Suppose that $S$ is a subset of the preferred basis for
  $\Dual{\KhCx(K)}$ so that the span of $S$ is a subcomplex of
  $\Dual{\KhCx(K)}$. Then there is a functor $F_{\Kh}^S\co \CCat{n}\to
  \BurnsideCat$ defined by 
  \begin{equation}\label{eq:F-subcx}
    \begin{split}
      F_{\Kh}^S(v)&=S\cap F_{\Kh}(v),\\
      F_{\Kh}^S(\cmorph{u}{v})&=s^{-1}(F_{\Kh}^S(u))\cap t^{-1}(F_{\Kh}^S(v))\subseteq F_{\Kh}(\cmorph{u}{v}),
    \end{split}
  \end{equation}
  with coherence maps induced by the coherence maps for
  $F_{\Kh}(K)$. Further, inclusion induces a natural transformation
  $\eta\from F_{\Kh}^S\to F_{\Kh}(K)$, and if $\Span(S)\to
  \Dual{\KhCx(K)}$ is a quasi-isomorphism then, for any $m\in\ZZ$,
  $\eta$ is a stable equivalence between the stable functors $\Sigma^m
  F_{\Kh}^S$ and $\Sigma^m F_{\Kh}(K)$. Similarly, if $S$ is a subset
  of the preferred basis for $\Dual{\KhCx(K)}$ so that the complement
  of $S$ spans a subcomplex of $\Dual{\KhCx(K)}$ then
  Formula~\eqref{eq:F-subcx} still defines a functor $F_{\Kh}^S\co
  \CCat{n}\to \BurnsideCat$, and there is a natural transformation
  $\eta\from F_{\Kh}(K)\to F_{\Kh}^S$ defined as follows. Recall that
  a natural transformation is a functor $\eta\co
  \CCat{1}\times\CCat{n}\to\BurnsideCat$. Define
  $\eta(\cmorph{1}{0}\times\Id_{v})=F_{\Kh}^S(v)=S\cap F_{\Kh}(v)$,
  with source map given by the inclusion $F_{\Kh}^S(v)\into
  F_{\Kh}(v)$ and target map the identity map $F_{\Kh}^S(v)\to
  F_{\Kh}^S(v)$; the definition extends in an obvious way to all of
  $\CCat{1}\times\CCat{n}$. If the quotient map
  $\Dual{\KhCx(K)}\to\Span(S)$ is a quasi-isomorphism then $\eta$ is a
  stable equivalence between the stable functors $\Sigma^m F_{\Kh}(K)$
  and $\Sigma^m F_{\Kh}^S$.

  To prove the theorem, it suffices to check invariance under the
  three Reidemeister moves. (For pointed links, we need to check
  invariance under the Reidemeister moves in the complement of the
  basepoint.) By the previous paragraph, it suffices to show that the
  (duals of the) isomorphisms on Khovanov homology induced by
  Reidemeister moves come from sequences of:
  \begin{itemize}[leftmargin=*]
  \item Inclusions of subcomplexes spanned by Khovanov generators,
    inducing isomorphisms on homology.
  \item Projections to quotient complexes spanned by Khovanov generators,
    inducing isomorphisms on homology.
  \item Face inclusions of cubes.
  \end{itemize}
  For Reidemeister I and II moves, Bar-Natan's
  formulation~\cite{Bar-kh-khovanovs} of Khovanov's invariance
  proof~\cite{Kho-kh-categorification} has this form, so proves
  invariance of the stable functor as well (see
  also~\cite[Propositions 6.2 and 6.3]{RS-khovanov}). For Reidemeister
  III moves, it suffices to prove invariance under the braid-like
  Reidemeister III move, which locally has the form
  $\sigma_1\sigma_2\sigma_1\sigma_2^{-1}\sigma_1^{-1}\sigma_2^{-1}\sim
  \Id$ (where the $\sigma_i$ are braid generators)~\cite[Section
  7.3]{Baldwin-hf-s-seq}. For this braid-like Reidemeister III, one
  can again reduce to the identity braid by a sequence of subcomplex
  inclusions and quotient complex projections, though the sequence is
  somewhat tedious~\cite[Proposition 6.4]{RS-khovanov}. The result follows.
\end{proof}

The following properties are strightforward from the definitions. (See
\cite[Proposition 9.1]{LLS-khovanov-product}.)
\begin{enumerate}[label=(X-\arabic*),ref=(X-\arabic*)]
\item\label{item:disj-union-Kh} If $K_1$ and $K_2$ are oriented links then \[F_{\Kh}^j(K_1\amalg K_2)\cong
  \coprod_{j_1+j_2=j}F_{\Kh}^{j_1}(K_1)\times F_{\Kh}^{j_2}(K_2).\]
\item If $(K_1,p_1)$ is a pointed, oriented link, and $K_2$ is an
  oriented link then \[F_{\rKh}^j(K_1\amalg K_2,p_1)\cong
  \coprod_{j_1+j_2=j}F_{\rKh}^{j_1}(K_1,p_1)\times F_{\Kh}^{j_2}(K_2).\]
\item If $(K_1,p_1)$ and $(K_2,p_2)$ are pointed, oriented links and
  $(K_1\#K_2,p)$ is the connected sum of $K_1$ and $K_2$ at the
  basepoints and the new basepoint $p$ is chosen on one of the
  connected sum strands, then \[F_{\rKh}^j(K_1\#K_2,p)\cong
  \coprod_{j_1+j_2=j} F_{\rKh}^{j_1}(K_1,p_1)\times
  F_{\rKh}^{j_2}(K_2,p_2).\]
\end{enumerate}

 \Section{spaces} describes a recipe for turning a stable functor
 $\Sigma^r F$ into a spectrum $\Realize{\Sigma^r F}$. Applying that
 recipe to the Khovanov functor associated to a link $K$
 (respectively, reduced Khovanov functor associated to a pointed link
 $(K,p)$) gives the Khovanov stable homotopy type $\KhSpace(K)$
 (respectively, $\rKhSpace(K,p)$).

\section{Spaces}\label{sec:spaces}
Finally, we return to the connection with topological spaces. Given a
diagram $F\co\CCat{n}\to\BurnsideCat$, one can associate an
essentially well-defined spectrum $\Realize{F}$. To give a concrete
construction of $\Realize{F}$, following
\cite[Section~5]{LLS-khovanov-product}, start by fixing an integer
$\ell\geq n$. We build a diagram $G\co \CCat{n}\to\CW$ of based CW
complexes, as follows. Given a vertex $v$, let
\[
G(v)=(F(v)\times [0,1]^\ell)/(F(v)\times \bdy([0,1]^\ell))\simeq \bigvee_{x\in F(v)}S^\ell.
\]
For each $v>w$, choose an embedding
\[
\Phi(\cmorph{v}{w})\co F(\cmorph{v}{w})\times[0,1]^\ell\to F(v)\times[0,1]^\ell
\]
satisfying the following conditions:
\begin{enumerate}[leftmargin=*]
\item For each $a\in F(\cmorph{v}{w})$,
  $\Phi(\{a\}\times[0,1]^\ell)\subseteq \{s(a)\}\times[0,1]^\ell$; that
  is, the following diagram commutes:
  \[
  \xymatrix@C=8ex{
    F(\cmorph{v}{w})\times[0,1]^\ell \ar[d] \ar@{^(->}[r]^-{\Phi(\cmorph{v}{w})}&F(v)\times[0,1]^\ell\ar[d]\\
    F(\cmorph{v}{w})\ar[r]^-s& F(v).
    }
  \]
\item For each $a\in F(\cmorph{v}{w})$, the embedding
  $(\{a\}\times[0,1]^\ell)\into (\{s(a)\}\times[0,1]^\ell)$ is an
  inclusion as a sub-box; that is, the lower arrow in the following
  diagram
  \[
  \xymatrix{
    \{a\}\times[0,1]^\ell \ar[d]^-\cong \ar@{^(->}[r]&\{s(a)\}\times[0,1]^\ell\ar[d]_-\cong\\
    [0,1]^\ell\ar@{^(->}[r]& [0,1]^\ell
    }
  \]
  is a map of the form
  \[
  (x_1,\dots,x_\ell)\mapsto (c_1+d_1x_1,\dots,c_\ell+d_\ell x_\ell)
  \]
  for some non-negative constants $c_1,d_1,\dots,c_\ell,d_\ell$.
\end{enumerate}
Define $G(\cmorph{v}{w})$ to be the composition
\[
F(v)\times [0,1]^\ell/\bdy\to F(\cmorph{v}{w})\times[0,1]^\ell/\bdy \to F(w)\times[0,1]^\ell/\bdy,
\]
where the $\bdy$ symbols denote the unions of the boundaries of the
cubes, the first map sends a point of the form $\Phi(a,x)$ to
$(a,x)\in F(\cmorph{v}{w})\times[0,1]^\ell$ and all other points to the
collapsed boundary, and the second map sends $(a,x)$ to $(t(a),x)$.

The resulting map $G\co \CCat{n}\to \CW$ does not commute on the nose:
the maps $G(\cmorph{v}{w})\circ G(\cmorph{u}{v})$ and
$G(\cmorph{u}{w})$ are (probably) defined using different choices of
embeddings. However, the spaces of embeddings $\Phi$ are
$(\ell-2)$-connected, so since any sequence of composable morphisms in
$\CCat{n}$ has length at most $n\leq\ell$, the cube commutes up to
coherent homotopies, i.e., is a \emph{homotopy coherent diagram} in
the sense of Vogt~\cite{Vogt-top-hocolim}. Homotopy coherent diagrams
are essentially as good as commutative diagrams: they appear
naturally whenever one takes a commutative diagram of spaces and
replaces all of the objects by homotopy equivalent ones, and
conversely every homotopy coherent diagram can be replaced by a
homotopy equivalent honestly commutative diagram essentially
canonically.

Next, we add a single object $*$ to the category $\CCat{n}$, and a
single morphism $v\to *$ for each non-terminal object $v$ of
$\CCat{n}$, to get a bigger category $\CCat{n}_+$. Extend $G$ to a
functor $G_+\co \CCat{n}_+\to\CW$ by declaring that $G_+(*)$ is a single
point. Finally, take the homotopy colimit of $G_+$ (as defined by Vogt
for homotopy coherent diagrams~\cite{Vogt-top-hocolim}) to define
\[
\CRealize[\ell]{F}=\hocolim(G_+).
\] 
Adding the object $*$ and taking the homotopy colimit is an iterated
version of the mapping cone construction. In particular, if
$G\from\CCat{1}\to\CW$ were a diagram on the one-dimensional cube,
then $G$ would consist of the data of a cellular map between two CW
complexes, $G(\cmorph{1}{0})\from G(1)\to G(0)$, and 
$\hocolim(G_+)$ would be the mapping cone of
$G(\cmorph{1}{0})$.  This iterated mapping cone of $G$ is a
space-level version of the totalization construction from
\Definition{totalization}. So, unsurprisingly, $\CRealize[\ell]{F}$
carries a CW complex structure so that its reduced cellular chain
complex $\cellC(\CRealize[\ell]{F})$ can be identified with
$\Sigma^\ell\Total(F)$, implying
$\wt{H}_\bullet(\CRealize[\ell]{F})\cong \Sigma^{\ell}
H_\bullet(\Total(F))$.

Similarly, given any natural transformation $\eta\from F\to F'$
between diagrams $F,F'\from\CCat{n}\to\BurnsideCat$, and an $\ell\geq
n+1$, one can construct a based cellular map
$\CRealize[\ell]{\eta}\from \CRealize[\ell]{F}\to\CRealize[\ell]{F'}$
so that the induced map on the reduced cellular chain complexes is the
map
$\Sigma^\ell\Total(\eta)\from\Sigma^\ell\Total(F)\to\Sigma^\ell\Total(F')$. This
is not functorial on the nose, since the construction depends on the
choices of embeddings $\Phi$. However, the space of choices of the
coherent homotopies in the construction of the homotopy coherent
diagrams $\CCat{n}\to\CW$ is an $(\ell-n-2)$-connected space, and the
space of choices in the construction of the map is an
$(\ell-n-1)$-connected space. By allowing $\ell$ to be arbitrarily
large, we can make these constructions \emph{essentially canonical},
namely, parametrized by a contractible space.

Therefore, to make the
space independent of $\ell$ and the choices of embeddings, and to make
the construction functorial, we replace the CW complex
$\CRealize[\ell]{F}$ by its suspension spectrum, and then formally
desuspend $\ell$ times:
\[
\Realize{F}=\Sigma^{-\ell}\CWtoSpec{\CRealize[\ell]{F}}.
\]
Alternately, one could replace the diagram $G$ by its suspension
spectrum before taking the homotopy colimit.  Finally for any stable
functor $\Sigma^r F$, define $\Realize{\Sigma^r F}$ to be the formal
suspension $\Sigma^r\Realize{F}$.

There are several other ways of turning a diagram
$F\co\CCat{n}\to\Spectra$ into a space:
\begin{enumerate}[leftmargin=*]
\item Produce a diagram $G\co\CCat{n}\to\Spectra$ by viewing $F$ as
  a functor $\wt{F}$ to the category of permutative categories, by
  letting $\wt{F}(v)=\Sets/F(v)$ be the category of sets over $F(v)$,
  and then applying $K$-theory. Then take the iterated mapping cone,
  as above. This is the procedure used by
  Hu-Kriz-Kriz~\cite{HKK-Kh-htpy}.
\item Turn $F$ into a (rather special) flow category, in the sense of
  Cohen-Jones-Segal, and then apply the Cohen-Jones-Segal
  realization~\cite{CJS-gauge-floerhomotopy}. This is essentially the
  approach taken in our previous work~\cite{RS-khovanov}.
\item\label{item:thickened} Without making any choices, produce a
  diagram $\thicf{F}\co\thic{\CCat{n}}\to\Spectra$ from a slightly
  larger category $\thic{\CCat{n}}$, and then take an appropriate
  iterated mapping cone of $\thicf{F}$~\cite[Section~4]{LLS-khovanov-product}.
\end{enumerate}
These constructions all give homotopy equivalent spectra~\cite[Theorem
3]{LLS-khovanov-product}. 

For convenience, we summarize some of the properties of this functor
\[
\Realize{\cdot}\from\BurnsideCat^{\CCat{n}}\to\Spectra.
\]
\begin{enumerate}[label=(Sp-\arabic*),ref=(Sp-\arabic*)]
\item\label{item:property-objects} For any stable functor
  $(F\from\CCat{n}\to\BurnsideCat,r)$,
  \begin{enumerate}[leftmargin=0ex, label=(\alph*),ref=(\alph*)]
  \item $\Realize{\Sigma^r F}$ is the formal suspension $\Sigma^r
    \Realize{F}$.
  \item\label{item:is-CW} The spectrum $\Realize{F}$ is homotopy
    equivalent to a formal desuspension of the suspension spectrum of
    some finite CW complex $\CRealize[\ell]{F}$.
  \item\label{item:homolo-ident} There is an identification of
    $\Total(F)$ with the cellular chain complex of
    $\Realize{F}$ from
    \ref{item:property-objects}\ref{item:is-CW}. In particular,
    $H_\bullet(\Realize{F})\cong
    H_\bullet(\Total(F))$ and
    $H^\bullet(\Realize{F})\cong
    H^\bullet(\Total(F))$.
  \end{enumerate}
\item\label{item:property-morphisms} For any natural transformation
  $\eta\from F\to F'$ between diagrams
  $F,F'\from\CCat{n}\to\BurnsideCat$, and any integer $r$,
  \begin{enumerate}[leftmargin=0ex, label=(\alph*),ref=(\alph*)]
  \item The map $\Realize{\Sigma^r\eta}$ is the formal suspension
    $\Sigma^r\Realize{\eta}$.
  \item\label{item:map-ident} The map $\Realize{\eta}$ is
    homotopic to a formal desuspension of some cellular map
    $\CRealize[\ell]{\eta}\from
    \CRealize[\ell]{F}\to\CRealize[\ell]{F'}$.
  \item\label{item:induced-map-homol} With respect to the
    identification from
    \ref{item:property-objects}\ref{item:homolo-ident} and
    \ref{item:property-morphisms}\ref{item:map-ident}, the map
    $\Total(\eta)$ can be identified with the map induced by
    $\Realize{\eta}$ on the reduced cellular chain complex.
  \end{enumerate}
\item\label{item:coprod-to-coprod} For any two functors $F,F'\co
  \CCat{n}\to\BurnsideCat$, there is a canonical homotopy equivalence
  $\Realize{F\amalg F'}\simeq\Realize{F}\vee\Realize{F'}$.
\item\label{item:prod-to-prod} For any two functors
  $F_1\from\CCat{n_1}\to\BurnsideCat$ and
  $F_2\from\CCat{n_2}\to\BurnsideCat$, there is a canonical homotopy
  equivalence $\Realize{F_1\times F_2}\simeq
  \Realize{F_1}\wedge\Realize{F_2}$.
\item\label{item:suspend-equiv} If $\iota\from\CCat{n}\into\CCat{N}$ is
  a face inclusion, then there is a canonical homotopy equivalence
  $\Realize{F}\simeq \Realize{\Sigma^{-\card{\iota}}F_\iota}$.
\item\label{item:equiv-equiv} In particular, from \ref{item:suspend-equiv} and
  \ref{item:property-morphisms}\ref{item:induced-map-homol}, stably
  equivalent stable functors give homotopy equivalent spectra.
\end{enumerate}
See~\cite{LLS-khovanov-product}, particularly Section~4, for more
details about the construction of $\Realize{F}$, and for proofs of
these properties.

\section{Some questions}
We end with a few questions. First, the Khovanov homotopy type $\KhSpace^j(K)$
contains more information than the Khovanov homology
$\Kh^{i,j}(K)=\wt{H}^i(\KhSpace^j(K))$. Specifically, the Khovanov spaces
$\KhSpace^j(K)$ induce Steenrod operations on $\Kh^{i,j}(K)$. One can
give a combinatorial formula for $\Sq^2\co \Kh^{i,j}(K;\F_2)\to
\Kh^{i+2,j}(K;\F_2)$~\cite{RS-steenrod}; using this, Seed showed
that there are knots with isomorphic Khovanov homologies but distinct
Steenrod squares~\cite{Seed-Kh-square}. Further, using $\Sq^2(K)$, one
can give refinements $s_{\pm}^{\Sq^2}$~\cite{RS-s-invariant} of
Rasmussen's $s$-invariant~\cite{Ras-kh-slice} and, using direct
computations and the connected sum formula from
Properties~\ref{item:disj-union-Kh} and~\ref{item:prod-to-prod}, one
obtains (modest) new concordance results~\cite[Corollary
1.5]{LLS-khovanov-product}.

The formula for $\Sq^2$ is combinatorial and not hard to formulate in
terms of $F_{\Kh}(K)$, but does not seem particularly obvious in this
language. We do not know how to generalize the formula to higher
Steenrod squares, or (reduced) $p\th$ powers.  So:
\begin{question}
  Are there nice formulations of the action of the Steenrod algebra on
  $\Kh(K)$, purely in terms of the stable Khovanov functor from
  \Definition{khovanov-functor}? Of other algebro-topological
  invariants of $\KhSpace(K)$ (such as the $K$-theory or bordism
  groups)? Do these give additional concordance invariants?
\end{question}

As described in \Section{spaces}, one can turn stable functors
$\CCat{n}\to\BurnsideCat$ into spectra. Perhaps this operation loses
useful information:
\begin{question}
  Are there stable functors $\Sigma^k F$ and $\Sigma^\ell G$ so that
  the spectra $\Realize{\Sigma^k F}$ and $\Realize{\Sigma^\ell G}$ are
  homotopy equivalent but $\Sigma^k F$ and $\Sigma^\ell G$ are not
  stably equivalent? If so, are there useful knot invariants that can
  be obtained from the Khovanov functor which are lost on passing to
  $\KhSpace(K)$?
\end{question}

There are other formulations of Khovanov homology, including
Cautis-Kamnitzer's formulation via algebraic
geometry~\cite{CK-kh-sheaves1}, formulations by
Webster~\cite{Webster-Kh-higher-rep,Webster-kh-tensor},
Lauda-Queffelec-Rose~\cite{LQR-kh-Howe}, and others via categorified
quantum groups, and Seidel-Smith's formulation via Floer
homology~\cite{SS-Kh-symp} (see also~\cite{AS-kh-agree-Q}). In the
Floer homology case, the Cohen-Jones-Segal
program~\cite{CJS-gauge-floerhomotopy} is expected to produce a Floer
spectrum.
\begin{question}
  Does the Cautis-Kamniter or Webster formulation of Khovanov homology
  have a natural extension along the lines of stable functors to the
  Burnside category?
\end{question}

\begin{question}
  Is the Floer spectrum (conjecturally) given by applying the
  Cohen-Jones-Segal construction to the Seidel-Smith formulation of
  Khovanov homology homotopy equivalent to $\KhSpace(K)$?
\end{question}

Finally, it would be nice to have an explicit description of the
stable Khovanov functor or the Khovanov homotopy type in more
cases. One possible direction would be to understand the stable
invariant in the sense of~\cite{Sto-kh-stable, GOR-kh-stable-torus}:
\begin{question}\label{q:stable}
  Do the Khovanov functors for $(m,N)$ torus knots admit a limit as
  $N\to\infty$, and if so, is there a simple description of the limit
  functor or the associated spectrum?
\end{question}

\begin{remark}
  Since the first version of this paper was written, Willis and
  Lobb-Orson-Schuetz have shown that the Khovanov homotopy types of
  torus knots stabilize (cf.\ Question~\ref{q:stable})~\cite{Willis-Kh-stable,LOS-Kh-colored}. In
  fact, they have deduced a more general stabilization phenomenon and
  deduced the existence of colored Khovanov stable homotopy
  types~\cite{LOS-Kh-colored,Willis-Kh-colored}.
\end{remark}

\bibliographystyle{myalpha}
\bibliography{newbibfile}

\end{document}